\documentclass[reqno]{amsart}
\usepackage{amssymb}
\usepackage{amsmath}
\usepackage{amsthm}
\usepackage[usenames]{color}
\usepackage{graphicx}
\usepackage{cite}
\usepackage{bbm}
\allowdisplaybreaks[4]

\usepackage[hidelinks]{hyperref}
\usepackage[margin=1.1in]{geometry} 
\usepackage{marginnote}

\usepackage{color}

\newtheorem{thm}{Theorem}[section]

\newtheorem{lem}[thm]{Lemma}
\newtheorem{prop}{Proposition}[section]
\newtheorem{rem}{Remark}[section]

\numberwithin{equation}{section}

\newcommand{\de}{\delta}

\newcommand{\vertiii}[1]{{\left\vert\kern-0.25ex\left\vert\kern-0.25ex\left\vert #1
		\right\vert\kern-0.25ex\right\vert\kern-0.25ex\right\vert}}

\makeatletter
\@namedef{subjclassname@2020}{%
	\textup{2020} Mathematics Subject Classification}
\makeatother

\begin{document}

\title[The compressible Euler-VFP system] 
{Low Mach number limit of the compressible Euler--Vlasov--Fokker--Planck system in the whole space}

\author[F. Li]{Fucai Li}
\address[F. Li]{School of Mathematics, Nanjing University, 
Nanjing 210093, P. R. China}
\email{fli@nju.edu.cn}

\author[J. Ni]{Jinkai Ni$^*$}  \thanks{$^*$\! Corresponding author}
\address[JKN]{School  of Mathematics, Nanjing University, Nanjing 
 210093, P. R. China}
\email{jinkaini123@gmail.com}

\author[Z. Zhang]{Zhipeng Zhang}   
\address[Z. Zhang]{School of Mathematical Sciences, 
Ocean University of China, Qingdao 266100, P. R. China}
\email{zhangzp@ouc.edu.cn}

\begin{abstract}
Although there are many important contributions on compressible and incompressible fluid-particle interaction models respectively, 
how to connect the two-type fluid-particle models via the low Mach number limit remains a challenging open problem. 
In this paper, we resolve it for the compressible isentropic fluid-particle  model (Euler--Vlasov--Fokker--Planck (Euler--VFP) system) in the whole space $\mathbb{R}^3$. 
First, we establish the  global-in-time {\it a priori estimates} of strong solutions that are uniform with respect to the Mach number $\varepsilon$ near the global Maxwellian. The proof relies on a refined energy method that combines the relaxation structure 
$b^\varepsilon-u^\varepsilon$ induced by the fluid-particle interaction and the symmetrized acoustic structure of the compressible Euler part in the model.
Under the assumption of well-prepared initial data, we derive a {\it global-in-time} uniform error estimate in the $H^2$ framework between the solution of the compressible Euler--VFP system and that of the limiting incompressible Euler--VFP system. A key point is to introduce the corrected acoustic variable
 $ q^\varepsilon-\varepsilon [P'(1)]^{-1}\pi,$
 which captures the pressure corrector in the low Mach number limit. This also allows us to exploit the exact cancellation of the singular acoustic terms and to close the {\it global-in-time} error estimate. The damping term $b^\varepsilon-u^\varepsilon$, which is absent in the pure Euler equations, plays an essential role in recovering the relative velocity dissipation and in controlling the coupled fluid-particle dynamics.
As a consequence, we prove the low Mach number limit of the compressible Euler--VFP  system with the convergence rate $\mathcal O(\varepsilon)$ in the time-continuous $H^2$ topology. In particular, the density fluctuation vanishes, while the velocity and the distribution function converge strongly to the corresponding solutions of the incompressible Euler--VFP system.

\end{abstract}

\date{\today}
	

\maketitle
\thispagestyle{empty}

\section{Introduction}

\subsection{Our model}
In this paper, we investigate the low Mach number limit of the compressible isentropic Euler--Vlasov--Fokker--Planck (Euler--VFP) system in the whole space $\mathbb{R}^3$. More precisely, 
let $F^\varepsilon=F^\varepsilon(t,x,v)\geq 0$ denote the particle distribution function at time \(t\geq 0\) and position \(x\in\mathbb{R}^3\) with velocity \(v\in\mathbb{R}^3\), and let
\(\rho^\varepsilon=\rho^\varepsilon(t,x)>0\) and \(u^\varepsilon = u^\varepsilon(t,x)\) denote the fluid density and velocity, respectively. 
The scaled compressible Euler--VFP system is given by (cf. \cite{LLNZ-2026-preprint,CG-CPDE-2006}):
\begin{equation}\label{I1}
\left\{
\begin{aligned}
&\partial_t\rho^\varepsilon
+\nabla_x\cdot(\rho^\varepsilon u^\varepsilon)=0,
\\ 
&\partial_t(\rho^\varepsilon u^\varepsilon)
+\nabla_x\cdot(\rho^\varepsilon u^\varepsilon\otimes u^\varepsilon)
+\frac1{\varepsilon^2}\nabla_xP(\rho^\varepsilon)
=
\int_{\mathbb R^3}(v-u^\varepsilon)F^\varepsilon {\rm d}v,\\
&\partial_t F^\varepsilon+v \cdot\nabla_x F^\varepsilon
=
\nabla_v\cdot\big[(v-u^\varepsilon)F^\varepsilon
+\nabla_v F^\varepsilon\big],
\end{aligned}
\right.
\end{equation}
where $\varepsilon\in(0,1)$ is the (scaled) Mach number, and the pressure function $P(\rho^\varepsilon)$ is given by $\gamma-$law, i.e.,
$$P(\rho^\varepsilon)=(\rho^\varepsilon)^{\gamma}\quad\text{with} \quad \gamma>1.$$
The pressure term \(\nabla_xP(\rho^\varepsilon)/\varepsilon^{2}\) becomes singular as \(\varepsilon\to0\), and it is exactly this term that enforces the incompressibility constraint in the limit.
It is necessary to distinguish between the low Mach number limit and  
 the small Deborah number limit to  fluid-particle flows. In the small Deborah number limit,  the particle transport and the Fokker--Planck relaxation are singularly scaled, and the limiting particle equation becomes a Smoluchowski or Kramer--Smoluchowski type equation (cf. \cite{FKW-2026,FQW-2024-M3AS}); while for the low Mach number limit, the particle equation in \eqref{I1} is maintained at its original kinetic scale. That is, the transport term \(v\cdot\nabla_xF^\varepsilon\), the interaction friction force term \((v - u^\varepsilon)F^\varepsilon\), and the particle velocity diffusion all remain of order one. 
Hence, the singular perturbation is only generated by the acoustic pressure \(\nabla_xP(\rho^\varepsilon)/\varepsilon^{2}\) in the fluid momentum equation \eqref{I1}$_2$, 
and the limiting system remains an incompressible Euler--VFP system with the kinetic particle dynamics preserved.

\subsection{Previous literature}
Before beginning our analysis on the system \eqref{I1}, we first review some results on incompressible/compressible  fluid-particle interaction models. 
Fluid-particle models date back to the pioneering contributions of O'Rourke \cite {O} and Williams \cite{W2}, and  arise in various industrial applications
including, for example,   spray dynamics\cite{BBB-05}, 
medical biosprays\cite{BBJ-05}, 
diesel engine operation \cite{R-52-1,R-52-2} 
and combustion theory \cite{W1,W2}.  
Mathmeatically,  the fluid-particle models usually consist  of an Vlasov or Vlasov--Fokker--Planck (VFP) equation for the particles phase
and  Euler/Navier--Stokes (NS) equations for the fluid phase. Now  there are 
substantial results on the  global existence and large-time behavior of solutions to  fluid-particle models.
We briefly mention a few of them. When the fluid part in  fluid-particle models is   homogeneous and incompressible,  
the global existence of weak solutions to the Stokes--Vlasov equations in a bounded domain was established by Hamdache \cite{Hamdache-1998}. 
Boudin et al. \cite{Boudin-2009} obtained global weak solutions of the NS--Vlasov system in 
\(\mathbb{T}^3\). 
Subsequently, this result was extended to bounded domain case  by Yu \cite{Yu-2013}, and   to time-dependent bounded domain case by Boudin, Grandmont and Moussa \cite{Boudin-2017}.  
Han-Kwan,   Moussa    and  Moyano  \cite{HMM-arma-2020} and  Danchin \cite{da} obtained   Fujita-Kato solutions   to  the Vlasov-Navier-Stokes system in   $\mathbb{T}^3$ and   $\mathbb{R}^3$ respectively. 
More recently, the uniqueness of weak solutions to the NS--Vlasov system in \(\mathbb{T}^2\) or \(\mathbb{R}^2\) was obtained in \cite{HMMM-RMI-2020}.
When taking into account the Brownian effect of the particles 
(i.e., \(\Delta_v F^\varepsilon\)), 
Goudon et al. \cite{Goudon-2010} proved the global existence of classical solutions  for the NS--VFP system with small initial data in \(\mathbb{T}^3\), 
while Carrillo, Duan and Moussa \cite{CDM-KRM-2011} studied the corresponding inviscid case. 
Later, Chae, Kang and Lee \cite{CKL-JDE-2011} established the global existence of weak solutions to the NS--VFP system in \(\mathbb{R}^d\) (\(d=2,3\)), as well as the existence and uniqueness of smooth solutions 
in \(\mathbb{R}^2\). 
When the fluid part is   inhomogeneous and incompressible, Wang and Yu \cite{WY-JDE-2015} constructed global weak solutions to the NS--Vlasov system with density-dependent drag force in a bounded domain, and Choi and Kwon \cite{CK-Nonlinearity-2015} studied the global existence and large-time behavior of strong solutions. Jiang, Li and Ni \cite{JLN-2025} investigated the global well-posedness and large-time behavior of solutions to the NS--VFP system. 
Li et al. \cite{LNSW-2025} investigated the global well-posedness, optimal decay rates, and inviscid limit to the NS--VFP system with density-dependent drag force and hence improved the results in \cite{CDM-KRM-2011}.

Moreover, 
extensive investigations have also been carried out on the compressible fluid-particle models. Mellet and Vasseur \cite{MV-MMMAS-2007} proved the global existence of weak solutions to the compressible NS--VFP system in a bounded domain with Dirichlet boundary conditions for the fluid velocity, and Dirichlet or reflection boundary conditions for the distribution function. Duan and Liu \cite{DL-KRM-2013} constructed the global existence and time decay of classical solutions 
to the compressible Euler--VFP  system in $\mathbb{R}^3$. 
In \cite{LMW}, Li, Mu and Wang established the global well-posedness of classical solutions within the $H^4$ framework under the assumption that the initial data is a small perturbation of the global Maxwellian, and obtained that the solutions decay at a rate of $(1+t)^{-1/2}$ in the $L^\infty$-norm. 
However, the decay behavior in the $L^2$-norm remains unaddressed in \cite{LMW}. 
Recently, Li, Ni and Wu \cite{LNW-ArXiv-2025} improved those results in \cite{LMW} by reducing the required regularity of the initial data from $H^4$ to $H^2$, and derived  optimal time-decay rates in both the $L^2$-norm and the $L^p$-norm with $2\le p \le6$.
Li, Ni and Wang \cite{LNW-2026} investigated the  global well-posedness and inviscid limit of the compressible NS--VFP system with density-dependent friction force in $\mathbb{R}^3$ and $\mathbb{T}^3$ respectively and hence improved the results in \cite{DL-KRM-2013}.
 Li, Liu and Yang \cite{LLY} gave the exponential convergence rate of the strong solutions to the compressible 
NS--VFP system with specular reflection boundary condition or the incompressible NS--VFP system with Maxwell boundary condition in a bounded domain, respectively. Li, Wang and Wang \cite{LWW} investigated the wave phenomena of the multi-dimensional compressible Euler/NS--VFP system.
Interested reader can also refer to \cite{MW,LNW-ArXiv-2024,BBBGLLM-esaim-2009} on recent progresses  on  the  compressible non-isentropic fluid-particle models.


Another significant topic in fluid-particle models is singular limit theory, and in particular, the hydrodynamic limit has drawn considerable attention recently, with a series of important progresses having been made.
Carrillo and Goudon \cite{CG-CPDE-2006} first studied the compressible Euler--VFP system, and formally derived hydrodynamic limits in the “bubbling” and “flowing” regimes under distinct scalings. Some rigorous justifications were achieved by Mellet and Vasseur \cite{MV08} and Ballew \cite{B-ZAMP-2020}, respectively, both of which employed the relative entropy approach. Similar strategy was also employed by Choi and Jung  to study the compressible NS--VFP system with a density-dependent viscosity \cite{CJ-2020}. 
The hydrodynamic limit results of the VFP equation coupled with the incompressible NS equations were established by Goudon, Jabin and Vasseur in
\cite{GJV-2004-2-IUMJ,GJV-IUMJ-2004}. 
Similar results have been extended to the inhomogeneous incompressible NS--VFP system by Su and Yao \cite{SY}. 
Recently, Han-Kwan and Michel \cite{HM-MAMS-2024} proved two classes of hydrodynamic limits for the incompressible NS--Vlasov system in $\mathbb{T}^3$, corresponding to the light particles and fine particle regimes.
Su et al.\cite{SWYZ-JDE-2023} extended some results of 
\cite{HM-MAMS-2024} to the inhomogeneous incompressible 
NS--Vlasov system.

In addition, the  low Mach number limit is one of the classical singular limits in fluid mechanics.
 For the  compressible isentropic Euler equations, extensive results have been established. 
Ebin \cite{E-Ann-1977} first studied the low Mach number limit for the isentropic compressible Euler equations with well-prepared initial data. 
 Later, Klainerman and Majda \cite{KM-CPAM-1981,KM-CPAM-1982} generalized this to symmetric hyperbolic systems. 
For general initial data case, by the scattering property of acoustic wave in the whole space, Ukai \cite{Ukai-JMKU-1986} and Asano \cite{Asano-87} proved the low Mach number limit of the compressible isentropic Euler equations in $\mathbb{R}^3$. Subsequently, Schochet \cite{Schochet-JDE-1994} and Joly, M\'etivier and Rauch \cite{JMR-ASENS-1995} investigated the limit in periodic domains by developing filtering methods. 
  Iguchi \cite{Ig-97}, Isozaki \cite{Is-87}, and Secchi \cite{se-2000}  investigated the incompressible limit to the compressible Euler equations  in   exterior domain,  $R^n_+$ and bounded domain, respectively.  
The low Mach number limit for the non-isentropic compressible  Euler equations is more challenging, primarily for two reasons. First, the entropy variable breaks the symmetry structure established in \cite{KM-CPAM-1981, KM-CPAM-1982}. Second, the coefficients of the associated acoustic wave equations depend on the entropy variable, which leads to intricate resonance phenomena. Consequently, only a few results are available, see \cite{Alazard-ADE-2005,Schochet-CMP-1986,MS-ARMA-2001}. Recently, Li et al. \cite{LNZZ-2026-arxiv} studied the low Mach number
limit for the compressible Navier--Stokes--Euler system in $\mathbb{R}^3$.


Similar to the low Mach number limit problem to  fluid equations which connects the compressible and incompressible 
fluid equations, it is natural and  of great significance to investigate the low Mach number limit of  compressible fluid-particle models to construct the relations between the compressible and incompressible fluid-particle models.
However, until to now, there are no results on  low Mach number limit to compressible fluid-particle models. The goal of this paper is to treat this issue, i.e.,  we shall  explore the low Mach number limit of the compressible Euler--VFP system \eqref{I1}.

\subsection{Main results} 

As mentioned before, our aim of this paper is to study the low Mach number limit of the compressible Euler--VFP system \eqref{I1} in $\mathbb{R}^3$. We first point out that although the Euler component alone in \eqref{I1}  is prone to shock formation, the relaxation mode $b^\varepsilon-u^\varepsilon$ (in the system   \eqref{A1} below) can suppresses this for small perturbations \cite{DL-KRM-2013,LNW-2026}. This observation motivates us 
to construct the low Mach number limit of the compressible Euler--VFP system \eqref{I1} in the framework of  small global solutions.   

We supplement the system \eqref{I1} with the following initial data
\begin{align}\label{int}
(\rho^\varepsilon,u^\varepsilon,F^\varepsilon)|_{t = 0}=(\rho_0^\varepsilon(x),u_0^\varepsilon(x),F^\varepsilon_0(x,v))\to (\bar{\rho},0,M(v))\quad\text{as}\quad|x|\to\infty.
\end{align}
Here, $M(v)$ is the global Maxwellian 
$(2\pi)^{-\frac{3}{2}}e^{-\frac{|v|^2}{2}}$,
and \(\bar\rho>0\) denotes a prescribed constant background density. Without loss of generality, we set $\bar\rho = 1$ throughout this paper.
By introducing the transformations
\begin{align*}
  \rho^\varepsilon=1+\varepsilon q^\varepsilon\quad \text{and}\quad F^\varepsilon=M+\sqrt{M}f^\varepsilon,   
\end{align*}
the  compressible Euler--VFP system \eqref{I1} with the initial data \eqref{int} can be rewritten as
\begin{equation}\label{A1}
\left\{
\begin{aligned}
&\partial_t  q^\varepsilon+u^\varepsilon\cdot\nabla_{x}q^\varepsilon+\frac{1+\varepsilon q^\varepsilon}{\varepsilon}\nabla_{x} \cdot u^{\varepsilon}=0,               \\
&\partial_t u^\varepsilon+u^\varepsilon\cdot\nabla_{x}u^\varepsilon+\frac{1}{\varepsilon} \frac{P^{\prime}(1+\varepsilon q^\varepsilon)}{1+\varepsilon q^\varepsilon} \nabla_{x}q^\varepsilon=\frac{b^\varepsilon-u^\varepsilon-a^\varepsilon u^\varepsilon}{1+\varepsilon q^\varepsilon},
\\ 
& \partial_t f^\varepsilon+v\cdot\nabla_{x} f^\varepsilon+ u^\varepsilon\cdot\nabla_{v}f^\varepsilon-\frac{1}{2}u^\varepsilon\cdot v f^\varepsilon-u^\varepsilon\cdot v\sqrt{M}=\mathcal{L}f^\varepsilon,\\
& (q^\varepsilon,u^\varepsilon,f^\varepsilon)|_{t=0}= (q_0^\varepsilon,u_0^\varepsilon,f_0^\varepsilon)=\Big(\frac{\rho_0^\varepsilon-1}{\varepsilon},
u_0^\varepsilon\frac{F_0^\varepsilon-M}{\sqrt{M}}\Big).
\end{aligned}
\right.
\end{equation}
Here $\mathcal{L}$  is the linearized Fokker--Planck operator defined by
\begin{align*}
\mathcal{L} f^\varepsilon=\frac{1}{\sqrt{M}}    \nabla_{v}\cdot\Big( M\nabla_{v} \Big(  \frac{f^\varepsilon}{\sqrt{M}} \Big)  \Big)=\Delta_{v}f^\varepsilon-\frac{|v|^2}{4}f^\varepsilon+\frac{3}{2}f^\varepsilon,
\end{align*}
and $a^\varepsilon,b^\varepsilon$  are the moments of $f^\varepsilon$ defined by
\begin{align*}
a^\varepsilon(t,x)
:=
\int_{\mathbb R^3}\sqrt M f^\varepsilon(t,x,v) {\rm d}v
\quad \text{and} \quad
b^\varepsilon(t,x)
:=
\int_{\mathbb R^3}v\sqrt M f^\varepsilon(t,x,v){\rm d}v.    
\end{align*}
Then, formally passing to the limit in \eqref{A1}$_1$ as $\varepsilon\rightarrow 0$, we obtain the incompressibility constraint
\begin{align*}
{\rm div}_x  u=0.    
\end{align*}
Moreover, the singular pressure term generates the limiting pressure gradient \(\nabla_x\pi\). Hence, the pair \((u,f)\) formally satisfies the following incompressible Euler--VFP system (cf.\cite{CDM-KRM-2011}):
\begin{equation}\label{A2}
\left\{
\begin{aligned}
&\partial_t u +u\cdot \nabla_{x} u+\nabla_{x}\pi=b-u-au,\\
& {\rm div}_x  u=0,
\\ 
& \partial_t f +v\cdot\nabla_{x} f+ u \cdot\nabla_{v}f -\frac{1}{2}u \cdot v f -u \cdot v\sqrt{M}=\mathcal{L} f,
\end{aligned}
\right.
\end{equation}
with the initial data
\begin{align}\label{A2-2}
(u,f)|_{t=0}=(u_0(x),f_0(x,v)) \rightarrow(0,0) ,\quad\text{as}\quad x\rightarrow+\infty,  
\end{align}
where 
\begin{align*}
a=\int_{\mathbb R^3} \sqrt{M}f{\rm d}v\quad \text{and} \quad b=    \int_{\mathbb R^3} v\sqrt{M}f{\rm d}v.
\end{align*}

Now, we are in a position to state our main result.

\begin{thm}[Uniform-in-$\varepsilon$ global estimate]\label{Th1}
Let $0<\varepsilon< 1$. Assume that the initial data $(q_0^\varepsilon, u_0^\varepsilon,f_0^\varepsilon )$ satisfy $\rho_0^\varepsilon=1+\varepsilon q_0^\varepsilon>0$, $F_0^\varepsilon=M+\sqrt{M}f^\varepsilon_0\geq 0$,  $(q^\varepsilon_0,u^\varepsilon_0)\in H^3$ and $f_0^\varepsilon\in L_v^2(H ^3)$, and there exists a constant $\varepsilon_0>0$ independent of $\varepsilon$ such that 
\begin{align}\label{TG1}
 \mathcal{X}_{ \varepsilon,0}:= \|(q_0^\varepsilon, u_0^\varepsilon)\|_{H^3}^2+  \|f_0^\varepsilon\|_{L_v^2(H^3)}^2\leq \varepsilon_0,
\end{align}
then the Cauchy problem of the compressible Euler--VFP system 
\eqref{A1}  admits a unique global strong solution
$(q^\varepsilon,u^\varepsilon,f^\varepsilon)$ satisfying 
$\rho^\varepsilon=1+\varepsilon q^\varepsilon>0 $, $F^\varepsilon=M+\sqrt{M}f^\varepsilon \geq 0$, and
\begin{align}\label{TG2}
&\sup_{t\geq 0} \big(  \|(q^\varepsilon,u^\varepsilon)(t)\|_{H^3}^2+ \|f^\varepsilon(t)\|_{L_v^2(H^3)} ^2 \big) + \int_0^\infty \sum_{|\alpha|\leq 3} \|\{\mathbf{I}-\mathbf{P}\}\partial^\alpha f^\varepsilon(\tau)\|_{\nu}^2 \,  {\rm d}\tau \nonumber\\
&\quad + \int_0^\infty \big( \|(b^\varepsilon-u^\varepsilon)(\tau)\|^2_{H^3}+\|\nabla_{x}(a^\varepsilon,b^\varepsilon,q^\varepsilon)(\tau)\|_{H^2}^2       \big)  \,   {\rm d}\tau\leq C_0 \mathcal{X}_{\varepsilon,0},
\end{align}
where $C_0 > 0$ is a constant independent of $\varepsilon$, time and initial data.

\end{thm}

In \cite{CDM-KRM-2011}, Carrillo, Duan and Moussa established the following result regarding the global well-posedness of the limit system \eqref{A2}:
\begin{prop}
[Global well-posedness of the limit system] \label{prop2}
Let $\nabla_{x} \cdot u_0(x)=0$ and $F_0=M+\sqrt{M}f_0\geq 0$ hold. Assume that the initial data $(u_0 ,f_0 )$ satisfy  
$u_0\in H^3$ and $f_0 \in L_v^2(H ^3)$, and 
there exists a constant $\varepsilon_1>0$ such that
\begin{align}\label{TGG1}
\mathcal{X}_0:=\|u_0\|_{H^3}^2+\|f_0\|_{L_v^2(H^3)}^2\leq \varepsilon_1,     
\end{align}
then the Cauchy problem of the incompressible Euler--VFP system 
\eqref{A2}--\eqref{A2-2} admits a unique global strong solution
$( u,f)$ satisfying 
$F =M+\sqrt{M}f \geq 0$, 
\begin{align}\label{TGG2}
&\sup_{t\geq 0}  \big( \| u(t)\|_{H^3}^2+ \|f (t)\|_{L_v^2(H^3)}^2\big)   + \int_0^\infty \sum_{|\alpha|\leq 3} \|\{\mathbf{I}-\mathbf{P}\}\partial^\alpha f (\tau)\|_{\nu}^2   \,{\rm d}\tau \nonumber\\
&\quad + \int_0^\infty \big( \|(b -u)(\tau)\|^2_{H^3}+\|\nabla_{x}(a ,b  )(\tau)\|_{H^2}^2       \big)  \,   {\rm d}\tau\leq C_1 \mathcal{X}_{ 0},
\end{align}
where $C_1 > 0$ is a constant independent of time and initial data.
Moreover, for any given $\delta>0$ which is close to zero, if $\|u_0\|_{ H^3\cap L^1}^2+\|f_0\|_{L_v^2(H^3\cap L^1)}^2\lesssim \varepsilon_1$, the solution $(u,f)$ enjoys the following time-decay:
\begin{align}\label{TGG3}
\|u(t)\|_{H^3}+\|f(t)\|_{L_v^2(H^3)}   \leq C_{\sigma} (1+t)^{-\frac{3}{4}+\sigma} \big(  \|u_0\|_{H^3\cap L^1} +\|f_0\|_{L_v^2(H^3\cap L^1)}   \big   ),
\end{align}
for any $t\geq 0$, where $C_{\sigma}$ depends only on $\sigma$ and may blow up as $\sigma$ tends to zero.
\end{prop}

With help of the estimate \eqref{TGG2}, we are able to establish the estimates of $\nabla_{x}\pi$ and $\nabla D_{t}\pi$, where $D_{t} :=\partial_{t}+u\cdot\nabla_{x}$,  which is of great significance in the subsequent proof of error estimate.

\begin{prop}[Estimate of the pressure function]\label{prop3}
Under the assumptions in Proposition \ref{prop2},
it holds that
\begin{align}\label{TGGG1}
\sup_{t\geq 0} \|\nabla_{x }\pi(t)\|_{H^2}\lesssim \varepsilon_1^\frac{1}{2},  
\end{align}
\begin{align}\label{TGGG2}
 \int_{0}^\infty \big(\|\nabla_{x}\pi(\tau)\|_{H^2}^2 
 +\|\nabla_{x}D_{t}\pi(\tau)\|_{H^1}^2  
\big)\,{\rm d}\tau  \lesssim \varepsilon_1,
\end{align}
\begin{align}\label{TGGG3}
 \int_{0}^\infty  \|D_{t}\pi(\tau)\|_{ \dot H^{-1}}^2  \,{\rm d}\tau  \lesssim \varepsilon_1,   
\end{align}
and
\begin{align}\label{TGGG4}
 \sup_{t\geq 0}  \| \pi(t)\|_{ L^2}   \lesssim \varepsilon_1^\frac{1}{2},   
\end{align}
where $\varepsilon_1$ is given by \eqref{TGG1}.
\end{prop}

\begin{rem}
The $L^1$-norm condition in Theorem \ref{prop2} is necessary because we need the decay estimate \eqref{TGG3} to ensure that \eqref{TGGG3} holds.
\end{rem}

\begin{rem}\label{Rem1.2}
The estimates
\eqref{TGGG1} and \eqref{TGGG4} indicate that $\pi\in C(\mathbb R^+;H^3)$, which is necessary in the subsequent proof of Theorem \ref{Th5}.
\end{rem}

\begin{rem}
The derivation of \eqref{TGG3} essentially used the coupling between the particle distribution and the fluid motion, which appears through the relaxation term $b-u$. This term yields a damping effect on the relative velocity and hence introduces a dissipative mechanism into the incompressible Euler--VFP system \eqref{A2}, see \eqref{G4.9}--\eqref{G4.11}. This feature is fundamentally different from the pure incompressible Euler equations, where no such particle-induced relaxation structure is available. 
\end{rem}

\begin{thm}[Global error estimate]\label{Th4}
Let $(q^\varepsilon,u^\varepsilon,f^\varepsilon)$  be the global solution to the Cauchy
problem of the scaled compressible Euler--VFP system \eqref{A1} 
given by Theorem \ref{Th1}, and $( u, \pi,f)$ be the corresponding global solution to the Cauchy problem of the
incompressible Euler--VFP system \eqref{A2} with the initial data $(u_0, f_0)$ stated in Proposition \ref{prop2}.
Assume further that
\begin{align}\label{TD1}
\eta_{\varepsilon}:= \|u^\varepsilon_0-u_0\|_{H^2}+   \|q^\varepsilon_0- \varepsilon [{P^{\prime}(1)}]^{-1} \pi_0\|_{H^2}+\|f_0^\varepsilon-f_0\|_{L_v^2(H^2)}\leq \varepsilon.
\end{align}
Then, there exists a constant $C_2> 0$ independent of time such that
\begin{align}\label{TD2}
&\sup_{t\geq 0}  \big( \|(u^\varepsilon-u)(t)\|_{H^2}^2+   \big\|\big(q^\varepsilon- \varepsilon [{P^{\prime}(1)}]^{-1} \pi\big)(t)\big\|_{H^2}^2+\|(f^\varepsilon-f)(t)\|_{L_v^2(H^2)}^2\big)   \nonumber\\
&\quad+\int_0^\infty \Big(\|\nabla_{x}(a^\varepsilon-a ,b^\varepsilon-b  )(\tau)\|_{H^1}^2  +\sum_{|\alpha|\leq 2} \big\|\{\mathbf{I}-\mathbf{P}\}\partial^\alpha (f^\varepsilon-f) (\tau)\big\|_{\nu}^2        \Big)  \,   {\rm d}\tau  \nonumber\\
&\quad\quad+\int_0^\infty \big( \big\|\big((b^\varepsilon -u^\varepsilon)-(b-u)\big)(\tau)\big\|^2_{H^2}+ \big\|\nabla_{x}\big(q^\varepsilon- \varepsilon [{P^{\prime}(1)}]^{-1} \pi\big)(\tau)\big\|_{H^1}^2 \big) \,   {\rm d}\tau\leq C_2 \varepsilon^2.
\end{align}
In \eqref{TD1}, $\pi_0$ is defined by
$\nabla\pi_0=\mathbb Q( -u_0\cdot\nabla_{x} u_0+b_0-u_0-a_0u_0)$ with  $\mathbb Qw:=-\nabla_{x}(-\Delta_{x})^{-1}{\rm div}_{x}  w$.   
\end{thm}

\begin{rem}
The corrected acoustic variable 
\begin{align*}q^\varepsilon-\varepsilon [P'(1)]^{-1}\pi    \end{align*} 
is introduced to eliminate the leading-order incompressible pressure contribution from the scaled density fluctuation.
This selection is in line with the well-prepared initial data and is crucial for the cancellation of the singular acoustic terms of order \(\varepsilon^{-1}\) in error estimates. It is exactly this correction that enables us to finalize the global-in-time \(H^2\) error estimate and achieve the convergence rate \(\mathcal O(\varepsilon)\).   
\end{rem}

\begin{rem}
The assumption \eqref{TD1} shows that the triple $(q_0^\varepsilon,u_0^\varepsilon,f_0^\varepsilon)$ is  well-prepared to order $\mathcal{O}(\varepsilon)$ in the $H^2$ framework because $\eta_{\varepsilon}$ approaches $0$ at the rate of $\mathcal{O}(\varepsilon)$ as $\varepsilon$ tends to $0$.
\end{rem}

\begin{rem}
Since ${\rm div}_x  u_0 = 0$, the inequality \eqref{TD1} also implies that $\mathbb{Q}u_0^\varepsilon=\mathbb{Q}(u_0^\varepsilon - u_0)\to 0$ in the $H^2$-norm. Therefore, the initial acoustic part, i.e., the density oscillation $q_0^\varepsilon-\varepsilon [P^\prime(1)]^{-1}\pi_0$, and the gradient part of the velocity $\mathbb{Q}u_0^\varepsilon$ vanish in the limit.
\end{rem}

\begin{rem}
In general, for ill-prepared initial data, the above strong convergence is no longer expected. Since the compressible solution may carry fast acoustic oscillations with a frequency of order $\mathcal O(\varepsilon^{-1})$, such oscillations cannot be removed by the uniform energy estimate alone. Therefore, one cannot expect
\begin{align*}
(q^\varepsilon,u^\varepsilon)\rightarrow&\,(0,u)\quad\text{strongly in }\quad C(\mathbb R^+;H^2),\\
f^\varepsilon \rightarrow&\, f\quad\quad\,\,\,\,\text{strongly in }\quad C(\mathbb R^+;L_v^2(H^2)).
\end{align*}
Instead, one may obtain weak-$*$ convergence of $u^\varepsilon$, strong convergence of the incompressible component $\mathbb P u^\varepsilon$ in suitable local spaces, or strong convergence after filtering out the fast acoustic waves. 
\end{rem} 

Based on the global error estimate presented in Theorem \ref{Th4}, we now provide the convergence analysis and identification for the limit of the compressible Euler-VFP system \eqref{A1}. We state the result as follows.

\begin{thm}[Low Mach number limit of the compressible Euler--VFP system]\label{Th5}
Let \(( q^\varepsilon,u^\varepsilon,f^\varepsilon)\) be the global solution to the
scaled compressible Euler--VFP system \eqref{A1}, and 
\((u,\pi,f)\) be the global small solution to the incompressible Euler--VFP system \eqref{A2}.
Assume that the assumptions of Theorem \ref{Th4} hold. Then, as \(\varepsilon\rightarrow 0\), we have
\begin{equation}\label{convergence}
\left\{
\begin{aligned}
(q^\varepsilon,\mathbb P u^\varepsilon,a^\varepsilon,b^\varepsilon) &\rightarrow (0,u,a,b) \,\,    \quad\text{strongly ~ ~~in ~~} C_{}(\mathbb{R}^{+};H^2),\\
\rho^\varepsilon-1&\rightarrow  0\quad
\quad \quad\quad\quad \text{strongly in ~~} C_{}(\mathbb{R}^{+};H^2)\cap L_{}^\infty(\mathbb R^+;L^\infty ), \\
\mathbb Qu^\varepsilon &\rightarrow 0 \quad \quad \quad\quad\quad\text{strongly in ~~} C_{}(\mathbb{R}^{+};H^2),\\
f^\varepsilon &\rightarrow f\quad \quad \quad  \quad\quad  \text{strongly in ~~}C_{}(\mathbb{R}^{+};L^2_v(H^2)). 
\end{aligned}
\right.
\end{equation}
In addition, the limit $(f,u,\pi)$ solves the incompressible Euler--VFP system \eqref{A2}--\eqref{A2-2} in the sense of distributions.
\end{thm}

\begin{rem}
 Our results still hold for  the compressible isentropic NS-VFP system (i.e., the system \eqref{I1} with an additional viscous term 
$-\mu\Delta u^\epsilon$ in \eqref{I1}$_2$, cf. \cite{MV-MMMAS-2007}). 
Comparing to the Euler-VFP system \eqref{I1}, in this case, the viscous term bring dissipation 
for the velocity field $u^\varepsilon$ and the whole arguments are much simpler, we omit the details here for simplicity.
\end{rem}
 
\begin{rem}
It is very interesting to  extend our results to the non-isentropic Euler-VFP system introduced in \cite{BBBGLLM-esaim-2009} which is our next research topic. In this situation, it is more challenging to analysis the dissipation effects of the fluid velocity and the fluid temperature.
\end{rem}

\subsection{Strategies in our proofs}
We now explain the main ideas in the proofs of Theorems \ref{Th1}, \ref{Th4}, and \ref{Th5}, as well as Proposition \ref{prop3}. First, we establish the uniform-in-$\varepsilon$ global estimate as stated in Theorem \ref{Th1}. The main difficulty stems from the singular acoustic terms in the scaled compressible Euler part in \eqref{A1}.
More precisely, the density and momentum equations \eqref{A1}$_1$--\eqref{A1}$_2$ contain the singular acoustic terms
\begin{align*}
\frac{1+\varepsilon q^\varepsilon}{\varepsilon}{\rm div}_x  u^\varepsilon\quad \text{and}\quad
\frac{1}{\varepsilon}
\frac{P'(1+\varepsilon q^\varepsilon)}
     {1+\varepsilon q^\varepsilon}
\nabla_x q^\varepsilon .    
\end{align*}
If one treats the fluid part only as a pure Euler system, these terms cannot be controlled by any parabolic smoothing  since the fluid equation is inviscid. And we can only obtain local-in-time estimates. Therefore,  
to overcome this difficulty, the key point is to exploit the intrinsic symmetric structure of the acoustic part.

Inspired by the standard symmetrization for the compressible Euler--VFP system \cite{CDM-KRM-2011}, during the high-order energy estimates, we multiply \eqref{G3.9}$_1$ by $\frac{P'(1+\varepsilon q^\varepsilon)}{(1+\varepsilon q^\varepsilon)^2}\partial_x^\alpha q^\varepsilon$ and \eqref{G3.9}$_2$ by \(\partial_x^\alpha u^\varepsilon\), respectively. 
After integration by parts, the two leading singular terms of order \(\varepsilon^{-1}\) cancel out precisely. The remaining terms involve spatial derivatives of the coefficient $\frac{P'(1+\varepsilon q^\varepsilon)}{1+\varepsilon q^\varepsilon}$, and the remaining commutator terms are treated as standard nonlinear terms (see Remark \ref{rem3.1}). 
Meanwhile, the fluid-particle coupling generates the relaxation term $b^\varepsilon - u^\varepsilon$, which offers a real damping mechanism that is lacking in the pure Euler equations. Motivated by \cite{LNW-2026},
the dissipation of \(u^\varepsilon\) is recovered through
\begin{align*}
\|\nabla u^\varepsilon\|_{H^2}\lesssim\|b^\varepsilon - u^\varepsilon\|_{H^3}+\|\nabla b^\varepsilon\|_{H^2}.
\end{align*}
Therefore, the relative velocity dissipation $b^\varepsilon-u^\varepsilon$, the moment equations for \(a^\varepsilon\) and \(b^\varepsilon\), and the symmetric cancellation together yield the uniform-in-\(\varepsilon\) global energy estimate.

Before proceeding with the global error estimate \eqref{TD2} in Theorem \ref{Th4}, we need to establish Proposition \ref{prop3}. On the one hand, we need to handle the linear or nonlinear terms concerning $D_{t}\pi$ and $\nabla_{x} \pi$ in the system \eqref{G5.2}. On the other hand, we also need $\pi \in C(\mathbb{R}^+; H^2)$ to prove the strong convergence in Theorem \ref{Th5}.
We first rewrite the pressure term through the Helmholtz decomposition: 
\begin{align*}
 \nabla_{x}\pi=\mathbb Q(b-u-au-u\cdot\nabla_{x}u).  
\end{align*}
Then, by meticulous calculation, we obtain the estimates \eqref{TGGG1} and \eqref{TGGG2}. The estimate  \eqref{TGGG3} is remarkable because we need to estimate $\|\partial_{t}\pi\|_{\dot H^{-1}}$. Compared with the $L^2$ estimate of $\nabla \pi$ in \eqref{TGGG1}, there is a loss of the second-order spatial derivative.
To handle this challenging term, we decompose the estimate of $\partial_{t}\pi$ into two parts (cf. \eqref{G4.8}), specifically:
\begin{align*} 
 \|\partial_{t}\pi\|_{\dot H^{-1}}\lesssim&\, \|  \partial_{t}(b-u) \|_{\dot H^{-1}}+\|  \partial_{t}(au)+\partial_{t}(u\cdot\nabla_{x}u)\|_{\dot H^{-1}} 
 =: \mathcal{I}_1+\mathcal{I}_2.
\end{align*}
For the term $\mathcal I_2$, given its nonlinear product property, we can handle it by employing Hardy--Littlewood--Sobolev's inequality $\|\Lambda^{-1}g\|_{L^2} \lesssim \|g\|_{L^{6/5}}  $ in Lemma \ref{L2.3}.
For the estimate of the linear term $\partial_{t}(b - u)$ in $\mathcal{I}_1$, we need to take into account some structural elements. The reason is that the equation regarding $b - u$ takes the form (see \eqref{G4.9}):
\begin{align} \label{LNZG1}
\partial_{t}(b-u) +2(b-u)=2au+u\cdot\nabla_{x} u+\nabla_{x}\pi -\nabla_{x}a-{\nabla_{x}}\cdot\Gamma (\{\mathbf{I}-\mathbf{P}\}f),
\end{align}
which contains the difficulty term $\nabla_{x}\pi$.  
To estimate \(b-u\) in \eqref{LNZG1}, we use the key observations derived from the equations for \(\mathbb P(b-u)\) and \(\mathbb Q b\) (see \eqref{G4.10}--\eqref{G4.11}):
\begin{align*} 
\partial_{t}(\mathbb P b-u)+2(\mathbb Pb-u)=&\, \mathbb P\big( 
2au+u\cdot\nabla_{x} u-\nabla_{x}a-{\rm div}_x \Gamma(\{\mathbf{I}-\mathbf{P}\}f)\big),\\
\partial_{t}\mathbb Q b+\mathbb Q b=&\,\mathbb Q\big(\nabla_{x}a-{\rm div}_x \Gamma (\{\mathbf{I}-\mathbf{P}\}f) +au\big).
\end{align*}
Both the above equations exhibit a damping structure, and the linear parts $\nabla_{x}a$ and ${\rm div}_x \Gamma(\{\mathbf{I}-\mathbf{P}\}f)$  are higher-order terms.
Therefore, by using the corresponding semi-group representation, together with the decay estimates \eqref{TGG3},
we are able to estimate  \(\mathcal I_1\), and thus prove \eqref{TGGG3}. Moreover, the estimate \eqref{TGGG4} follows from Lemma \ref{L2.3} combined with the above estimates. Consequently, we complete the proof
of Proposition \ref{prop3}.

Next, under the well-prepared initial condition \eqref{TD1}, we prove the global \(H^2\)-error estimate stated in Theorem \ref{Th4}. To rigorously justify the global-in-time convergence in the low Mach number limit, it is necessary to obtain a uniform convergence rate by comparing the scaled compressible system \eqref{A1} with the limiting incompressible Euler-VFP system \eqref{A2}. The key point is to introduce the corrected acoustic variable
\begin{align*}
\delta q := q^\varepsilon - \varepsilon [P'(1)]^{-1}\pi,
\end{align*}
where \(\pi\) denotes the limiting incompressible pressure in \eqref{A2}$_1$. 
This correction absorbs the contribution of the leading-order pressure in the acoustic variable and removes the forcing term of order \(\varepsilon^{-1}\), leaving only the source terms of order \(\varepsilon\) and quadratic. We then construct the   energy functional $\delta\mathcal{X}(t)$ as follows (see \eqref{G5.1}):
\begin{align*} 
 \delta\mathcal{X}(t):=&\,\sup_{\tau\in [0,t]}\big(\|(\delta q,\delta u)(\tau)\|_{H^2}^2 +  \|\delta f(\tau)\|_{L_v^2(H^2)}^2\big)+\int_0^t  \sum_{|\alpha|\leq 2} \|\{\mathbf{I}-\mathbf{P}\}\delta f(\tau)\|_{\nu}^2\,{\rm d}\tau \nonumber\\
 &+\int_0^t \big(\|\nabla_{x}(\delta a,\delta b,\delta q)(\tau)\|_{H^1}^2+ \|\delta (b-u)(\tau)\|_{H^2}^2    \big)\,{\rm d}\tau.
\end{align*}
Another difficulty comes from the source term of the pressure correction. 
Indeed, the source part \(\delta F_1\) (see \eqref{ff1}$_1$) contains the term $\varepsilon [P^{\prime}(1)]^{-1} D_{t}\pi$, which cannot be treated by a direct \(L^2\)-estimate of \(\delta q\).
Instead, we use the estimate of \(D_t\pi\) obtained in Proposition \ref{prop3}. More precisely, by the duality between \(\dot H^{-1}\) and \(\dot H^1\), we have
\begin{align*}
\varepsilon\int_0^t\!\!\int_{\mathbb R^3}|D_t\pi|\,|\delta q|\,{\rm d}x{\rm d}\tau
\lesssim
\varepsilon \|D_t\pi\|_{L_t^2(\dot H^{-1})}
\|\delta q\|_{L_t^2(\dot H^1)}\lesssim\varepsilon^2+\varepsilon_1\delta\mathcal{X}(t).    
\end{align*}
By combining the refined energy method, the macro-micro decomposition technique, and the estimates of \(D_t\pi\) and \(\nabla_x\pi\) obtained in Proposition \ref{prop3}, we prove Lemmas \ref{L5.1}--\ref{L5.4}. 
Subsequently, we obtain the following {\it global-in-time} error bound:
\begin{align*}
\delta\mathcal{X}(t)\lesssim \varepsilon^2,
\end{align*}
which gives the convergence rate \(\mathcal{O}(\varepsilon)\) and completes the proof of Theorem \ref{Th4}.

Finally, we demonstrate the strong convergence and limiting processes  as in Theorem \ref{Th5}. 
By virtue of the global error estimate \eqref{TD2},
we first obtain the strong convergence of the solution components. The only point requiring an additional observation is the convergence
\begin{align*}
 \mathbb Qu^\varepsilon\rightarrow 0\quad \text{in} \quad C(\mathbb R^+;H^2)\quad \text{as}\quad \varepsilon\rightarrow 0,
\end{align*}
which relies on the incompressibility condition $\nabla_x\cdot u=0$. 
To derive the incompressibility of the limit velocity  $u$, we multiply \eqref{A1}$_1$ by \(\varepsilon\) and obtain
\begin{align}\label{LNZG2}
\rho^\varepsilon {\rm div}_x  u^\varepsilon=-\varepsilon \partial_{t}q^\varepsilon-\varepsilon u^\varepsilon\cdot \nabla_{x}q^\varepsilon.     
\end{align}
Since \(\varepsilon q^\varepsilon\to0\) in \(C([0,T];H^2)\), the right-hand side of \eqref{LNZG2} tends to \(0\) in \(\mathcal D'\). On the other hand, the left-hand side converges to \(\nabla_x\cdot u\), which implies that ${\rm div}_x  u=0$ in $\mathcal{D}^\prime$. 
Then, by using the boundedness of \(\mathbb Q\) on \(H^2\), we obtain
\begin{align*}
  \mathbb Q u^\varepsilon= \mathbb Q(u^\varepsilon-u)+\mathbb Q u=\mathbb Q(u^\varepsilon-u)
\to 0 \quad\text{in}\quad C([0,T];H^2)\quad \text{as}\quad \varepsilon\rightarrow0. 
 \end{align*}
 
When treating the momentum equation \eqref{A1}$_2$, one has to be careful with the singular acoustic term
\begin{align*}
\Lambda^\varepsilon:=\varepsilon^{-1}\frac{P^\prime(1+\varepsilon q^\varepsilon)}{1+\varepsilon q^\varepsilon}\nabla_x q^\varepsilon,
\end{align*}
which is of order \(O(\varepsilon^{-1})\) and cannot be passed to the limit directly.
Fortunately, \(\Lambda^\varepsilon\) is the gradient of a function. More precisely,
\begin{align*}
\Lambda^\varepsilon=\nabla_x\Pi^\varepsilon, \quad
\Pi^\varepsilon
=
\varepsilon^{-2}\big(h(\rho^\varepsilon)-h(1)\big),
\quad \text{with}\quad
h'(r)=\frac{P'(r)}{r}.    
\end{align*}
We therefore apply the Leray projection \(\mathbb P\) to the equation \eqref{A1}$_2$. Since \(\mathbb P\nabla_x = 0\), the singular acoustic term is eliminated at the equation level. The projected equation is free of the \(O(\varepsilon^{-1})\) term and can be taken to the limit by using the strong convergence obtained above and the boundedness of \(\mathbb P\) on \(H^2\), which results in the following equation:
\begin{align*}
\partial_t u+\mathbb{P}(u\cdot\nabla u)=\mathbb{P}(b - u - au).
\end{align*}
Thus, the limiting momentum equation \eqref{A2}$_1$ holds in the sense of distributions.
It remains to take the limit in the kinetic equation \eqref{A1}$_3$.  The products $u^\varepsilon\cdot\nabla_v f^\varepsilon$ and $(u^\varepsilon\cdot v)f^\varepsilon$ converge because $u^\varepsilon\to u$ uniformly on the compact support of the test function while $f^\varepsilon\to f$.
The collision term is handled by the self-adjointness of $\mathcal L$. Consequently, we identify the incompressible Euler--VFP system \eqref{A2} in the sense of distributions, 
thereby completing the proof of Theorem \ref{Th5}.

\subsection{Outline of this paper}
The remainder of this paper is organized as follows. In Section 2, we introduce some notations, basic analytical tools, and several useful lemmas that will be used throughout this paper. In Section 3, we establish uniform-in-$\varepsilon$ global {\it a priori} estimates for the scaled compressible Euler--VFP system \eqref{A1} and prove Theorem \ref{Th1}. Section 4 is devoted to the analysis of the   incompressible Euler--VFP system \eqref{A2}, with particular emphasis on the pressure estimates required for the subsequent error analysis. In Section 5, we justify the global-in-time low Mach number limit. More precisely, in Subsection 5.1, we prove a global $H^2$-error estimate under {\it well-prepared initial data}; the corrected acoustic variable and cancellation of singular acoustic terms are the key ingredients. Finally, in Subsection 5.2, we combine the global $H^2$-error estimate with uniform-in-$\varepsilon$ bounds to pass to the limit, identify the limiting incompressible Euler--VFP system \eqref{A2}, and establish strong convergence in the time-continuous $H^2$ topology with the convergence rate $\mathcal{O}(\varepsilon)$.

\section{Preliminaries}

\subsection{Notations}
Throughout this paper, \(C\) denotes a generic positive constant which is independent of the time variable \(t\) and may vary from one line to another. For two non-negative quantities \(A\) and \(B\), we write \(A\lesssim B\) if there exists a constant \(C > 0\) such that \(A\leq CB\). We also write \(A\sim B\) if both \(A\lesssim B\) and \(B\lesssim A\) are satisfied; equivalently, \(C^{-1}A\leq B\leq CA\). If \(g,h\in X\), where \(X\) is a Banach space, we adopt the convention \(\|(g,h)\|_{X}:=\|g\|_{X}+\|h\|_{X}\). Additionally, \([\cdot,\cdot]\) represents the commutator \([A,B]=AB - BA\) for two operators \(A\) and \(B\).

For functions that depend only on \(v\), we define  
\begin{align*}
|g|_{\nu}^{2}:=\int_{\mathbb R^3_v}\left(|\nabla_v g(v)|^2+\nu(v)|g(v)|^2\right) {\rm d}v,
\end{align*}
where the velocity weight is defined as
\begin{align*}
\nu(v):=1+|v|^2.
\end{align*}
For functions depending on both \(x\) and \(v\), we set
\begin{align*}
\|g\|_{\nu}^{2}
:=
\int_{\mathbb R^3_x}\!\!\int_{\mathbb R^3_v}
\left(
|\nabla_v g(x,v)|^2+\nu(v)|g(x,v)|^2
\right) {\rm d}v{\rm d}x .    
\end{align*}
For a multi-index \(\alpha=(\alpha_1,\alpha_2,\alpha_3)\in\mathbb N^3\), we write
\[
\partial^\alpha=\partial_x^\alpha
:=
\partial_{x_1}^{\alpha_1}
\partial_{x_2}^{\alpha_2}
\partial_{x_3}^{\alpha_3},
\qquad
|\alpha|:=\alpha_1+\alpha_2+\alpha_3 .
\]
For simplicity, \(\partial_i\) is used to denote \(\partial_{x_i}\), where \(i = 1,2,3\). For an integer \(m\geq0\), we adopt the following Sobolev-type norms:
\begin{align*}
\|g\|_{H^m}
:=
\sum_{|\alpha|\leq m}
\|\partial_x^\alpha g\|_{L^2_x},\quad \|g\|_{\dot H^m}
:=
\sum_{|\alpha|= m}
\|\partial_x^\alpha g\|_{L^2_x}
\quad \text{for}\quad g=g(x),    
\end{align*}
and
\begin{align*}
\|g\|_{L_v^2(H^m)}
:=
\sum_{|\alpha|\leq m}
\|\partial_x^\alpha g\|_{L^2_{x,v}}, \quad  \|g\|_{L_v^2(\dot H^m)}
:=
\sum_{|\alpha|= m}
\|\partial_x^\alpha g\|_{L^2_{x,v}}
\quad\text{for}\quad g=g(x,v).    
\end{align*}
Finally, we use $\langle\cdot,\cdot\rangle$ to denote the inner product over the Hilbert space $L_v^2$, i.e., for any $g,h \in L^2_v$,
\begin{align*}
\langle g,h\rangle := \int_{\mathbb{R}^3_v} g(v) h(v) \mathrm{d}v.
\end{align*}

Next, we introduce the macro-micro decomposition and certain properties of $\mathcal{L}$.
\subsection{Macro-micro decomposition}
Motivated by \cite{DFT-2010-CMP,GY-iumj-2004}, we decompose $f(t,x,v)$ as the sum of the fluid part $\mathbf{P}f$ and the particle part $(\mathbf{I} - \mathbf{P})f$, i.e.,
\begin{align}\label{G2.1}
f(t,x,v) = \mathbf{P}f + (\mathbf{I} - \mathbf{P})f.
\end{align}
Here, the velocity orthogonal projection $\mathbf{P}$ is defined by $\mathbf{P}=\mathbf{P_0}\oplus \mathbf{P}_{1}$, where $\mathbf{P}_0f=a\sqrt{M}$ and $\mathbf{P}_1f=b\cdot v\sqrt{M}$, and it satisfies  
\begin{equation*}
\mathbf{P}:  L_{v}^{2}(\mathbb{R}^3)\rightarrow \mathrm{Span}\{\sqrt{M}, v_1 \sqrt{M}, v_2\sqrt{M},v_3\sqrt{M}\}.
\end{equation*}
In view of \eqref{G2.1}, we can decompose $\mathcal{L} f$ as 
\begin{align*}  
\mathcal{L}f = \mathcal{L}\{\mathbf{I}-\mathbf{P}\}f + \mathcal{L}\mathbf{P}f = \mathcal{L}\{\mathbf{I}-\mathbf{P}\}f - \mathbf{P}_1f.  
\end{align*}  
By the self-adjointness and coercivity of the Fokker--Planck operator \(\mathcal L\), there exists a constant \(\lambda_0>0\) such that
\begin{align}  \label{G2.2}
\langle -\mathcal{L} \{\mathbf{I}-\mathbf{P}\}f, f \rangle  
= \langle -\mathcal{L} \{\mathbf{I}-\mathbf{P}\}f, \{\mathbf{I}-\mathbf{P}\}f \rangle  
\geq \lambda_0 \left| \{\mathbf{I}-\mathbf{P}\}f \right|_{\nu}^2.
\end{align}  
Therefore,  by \eqref{G2.1}, we further obtain
\begin{align}\label{G2.3}  
\langle -\mathcal{L} f, f \rangle \geq \lambda_0 \left| \{\mathbf{I}-\mathbf{P}\}f \right|_{\nu}^2 + |b|^2.  
\end{align}
Thus, in accordance with \cite{AMTU-CPDE-2001}, the Fokker--Planck operator \(\mathcal L\) provides dissipation on both the microscopic component $\{\mathbf{I}-\mathbf{P}\}$ and the macroscopic momentum mode $b$.

\subsection{Some useful lemmas}
Finally, we present several useful lemmas that are frequently utilized throughout this paper.
\begin{lem} [{\!\!\cite[Lemma 2.1]{CDM-KRM-2011}} and {\cite[Lemmas 2.1--2.2]{Dk-MZ-1992}}]\label{L2.1}   
For any $g,h\in H^3(\mathbb{R}^3)$ and any multi-index $\alpha$  with $1\leq|\alpha|\leq3$, it holds that
\begin{align*}
\|g\|_{L^{\infty} (\mathbb{R}^3)} \lesssim&\, \|\nabla_{x}  g\|_{L^{2}(\mathbb{R}^3)}^{\frac{1}{2}}
\|\nabla_{x}^{2}g\|_{L^{2} (\mathbb{R}^3)}^{\frac{1}{2}}, \\
\|gh\|_{H^{1}(\mathbb{R}^3) }  \lesssim&\, \|g\|_{H^{2}(\mathbb{R}^3) }\|\nabla_{x}  h\|_{H^{2}(\mathbb{R}^3) }, \\
\|\partial^{\alpha}_{x}(gh)\|_{L^{2}(\mathbb{R}^3) }
\lesssim&\,\|\nabla_{x}  g\|_{H^{2}(\mathbb{R}^3) }\|\nabla_{x}  h\|_{H^{2}(\mathbb{R}^3) },\\
\|g\|_{L^6(\mathbb{R}^3)} \lesssim&\,  \|\nabla_{x}  g\|_{L^2(\mathbb{R}^3)}\lesssim  \|g\|_{H^1 (\mathbb{R}^3)},\\
  \|g\|_{L^q (\mathbb{R}^3)} \lesssim&\, \|g\|_{H^1(\mathbb{R}^3) }, \quad 2\leq q\leq 6.
\end{align*}
\end{lem}

\begin{lem}[{{\!\!\cite{commutator1}} and \cite[Appendix]{commutator2}}]\label{L2.2}
Let $g$ and $h$ be Schwarz functions. For $k\geq 0$, we have
\begin{align*}
\|\nabla_{x}^{k}(gh) \|_{L^r(\mathbb{R}^3)} \lesssim&\, \|g\|_{L^{r_1}(\mathbb{R}^3) }\|\nabla^{k}_{x}h\|_{L^{r_2} }+ \|h\|_{L^{r_3}(\mathbb{R}^3) }\|\nabla^{k}_{x}g\|_{L^{r_4} (\mathbb{R}^3)},\\
\|\nabla^{k}_{x}(gh)-g\nabla^k_{x} h \|_{L^r(\mathbb{R}^3)} \lesssim&\, \|\nabla_{x} g\|_{L^{r_1}(\mathbb{R}^3)}\|\nabla_{x}^{k-1}h\|_{L^{r_2}(\mathbb{R}^3)}+ \|h\|_{L^{r_3}(\mathbb{R}^3)}\|\nabla^{k}_{x}g\|_{L^{r_4}(\mathbb{R}^3)},    
\end{align*}
where $1<r,r_2,r_4<\infty$ and $r_i(1\leq i\leq 4)$ satisfy 
\begin{align*}
\frac{1}{r_1}+\frac{1}{r_2}=\frac{1}{r_3}+\frac{1}{r_4}=\frac{1}{r}.   
\end{align*}
\end{lem}

\begin{lem}[{\!\!\cite[ Theorem 1]{Stein-1970}}] \label{L2.3}
Let $0<s<3$, $1<p<q<\infty$, $\frac{1}{q}+\frac{s}{3}=\frac{1}{p}$, then
\begin{align*}
\|\Lambda^{-1} g\| _{L^q}\lesssim \|g\|_{L^p}.   
\end{align*}

\end{lem}

\begin{lem}[\!\!\cite{AF-Pa-2003}]\label{L2.4}
Suppose that $1\leq r\leq s\leq q\leq \infty$, and
\begin{align*}
\frac{1}{s}=\frac{\zeta}{r}+\frac{1-\zeta}{q},
\end{align*}
where $0\leq \zeta\leq 1$.
Assume also $g\in L^r(\mathbb{R}^3)\cap L^q(\mathbb{R}^3)$. Then $g\in L^s(\mathbb{R}^3)$, and
\begin{align*}
\|g\|_{L^s}\leq \|g\|_{L^r}^\zeta \|g\|_{L^q}^{1-\zeta}.    
\end{align*}
\end{lem}

\section{Uniform global well-posedness}

This section is devoted to proving the global well-posedness of the Cauchy problem \eqref{A1}, along with the derivation of estimates which are uniform in $\varepsilon$.

\subsection{Uniform global a priori estimates}
{Now, we proceed to establish 
global a priori estimates
for $(q^\varepsilon,u^\varepsilon,f^\varepsilon)$ under the assumption that
\begin{align}\label{G3.1}
\sup_{0\leq t\leq T}  \big\{ \|(q^\varepsilon,u^\varepsilon)(t)\|_{H^3} + \|f^\varepsilon(t)\|_{L_v^2(H^3)} \big\} \leq \sigma ,
\end{align}
where $0 < \sigma < 1$ is a sufficiently small constant independent of $\varepsilon$. Here, $(q^\varepsilon, u^\varepsilon, f^\varepsilon)$ denotes the strong solution to the scaled compressible Euler--VFP system \eqref{A1} on the time interval $[0,T]$, for some given $T > 0$.}
To begin with, we establish the zero-order estimate for $(q^\varepsilon,u^\varepsilon,f^\varepsilon)$.

\begin{lem}\label{L3.1}
For the strong solution $(q^\varepsilon, u^\varepsilon, f^\varepsilon)$ to the compressible Euler--VFP system \eqref{A1}, there exists a positive constant $\lambda_1$ independent of $\varepsilon$, such that
\begin{align}\label{G3.2}
&\frac{{\rm d}}{{\rm d}t} \big( P^\prime(1)\|q^\varepsilon\|_{L^2}^2+\|u^\varepsilon\|_{L^2}^2+\|f^\varepsilon\|_{L_v^2(L^2)}^2     \big) \nonumber\\ 
&\quad +\lambda_1\big( \|\{\mathbf{I}-\mathbf{P}\}f^\varepsilon\|_{\nu}^2+\|b^\varepsilon-u^\varepsilon\|_{L^2}^2 \big)
\lesssim \sigma \|\nabla_{x}(q^\varepsilon,u^\varepsilon,a^\varepsilon,b^\varepsilon)\|_{L^2}^2,  
\end{align}
for any $0\leq t<T$.    
\end{lem}

\begin{proof}
Multiplying \eqref{A1}$_1$--\eqref{A1}$_3$ by $P^\prime(1) q^\varepsilon$, $u^\varepsilon$, and $f^\varepsilon$ respectively, then integrating and summing, we obtain
\begin{align}\label{G3.3}
&\frac{1}{2}\frac{{\rm d}}{{\rm d}t} \big( P^\prime(1)\|q^\varepsilon\|_{L^2}^2+\|u^\varepsilon\|_{L^2}^2+\|f^\varepsilon\|_{L_v^2(L^2)}^2\big)+  \int_{\mathbb R^3} \langle  -\mathcal{L} \{\mathbf{I}-\mathbf{P}\}f^\varepsilon , f^\varepsilon\rangle \,{\rm d}x+ \|b^\varepsilon-u^\varepsilon\|_{L^2}^2 \nonumber\\
=&\, \frac{1}{\varepsilon}\int_{\mathbb R^3}  \Big( \frac{P^\prime(1+\varepsilon q^\varepsilon)}{1+\varepsilon q^\varepsilon}-P^\prime(1)    \Big)q^\varepsilon {\rm div}_x  u^\varepsilon\,{\rm d}x+\frac{1}{\varepsilon}\int_{\mathbb R^3}q^\varepsilon\nabla_{x}\Big( \frac{P^\prime(1+\varepsilon q^\varepsilon)}{1+\varepsilon q^\varepsilon}-P^\prime(1)       \Big)\cdot u^\varepsilon\,{\rm d}x\nonumber\\
&-\int_{\mathbb R^3}\frac{\varepsilon q^\varepsilon\big((b^\varepsilon-u^\varepsilon)\cdot u^\varepsilon-a^\varepsilon u^\varepsilon\cdot u^\varepsilon\big)} {1+\varepsilon q^\varepsilon} \,{\rm d}x-\frac{P^\prime(1)}{2} \int_{\mathbb R^3} (q^\varepsilon)^2 {\rm div}_x  u^\varepsilon\,{\rm d}x\nonumber\\
&+\frac{1}{2}\int_{\mathbb R^3} |u^\varepsilon|^2{\rm div}_x  u^\varepsilon \,{\rm d}x+ \frac{1}{2} \int_{\mathbb R^3} u^\varepsilon\cdot \langle vf^\varepsilon , f^\varepsilon\rangle\,{\rm d}x-\int_{\mathbb R^3}a^\varepsilon u^\varepsilon\cdot u^\varepsilon\,{\rm d}x \nonumber\\
\equiv:&\, \sum_{j=1}^7 I_{j}.
\end{align}
For the first two singular terms $I_1$ and $I_2$, applying H\"{o}lder’s and Sobolev’s inequalities along with \eqref{G3.1}, we have
\begin{align}\label{G3.4}
I_1+I_2\lesssim&\, \frac{1}{\varepsilon} \|\varepsilon q^\varepsilon\|_{L^6}\|q^\varepsilon\|_{L^3}\|\nabla_{x} u^\varepsilon\|_{L^2}+  \frac{1}{\varepsilon} \|q^\varepsilon\|_{L^3}\|\varepsilon\nabla_{x}q^\varepsilon\|_{L^2} \|u^\varepsilon\|_{L^6}\nonumber\\
\lesssim&\, \|(q^\varepsilon,u^\varepsilon)\|_{H^1} \|\nabla_{x}(q^\varepsilon,u^\varepsilon)\|_{L^2}^2\nonumber\\
\lesssim&\, \sigma \|\nabla_{x}(q^\varepsilon,u^\varepsilon)\|_{L^2}^2.
\end{align}
Since the embedding $H^2(\mathbb{R}^3) \hookrightarrow L^\infty(\mathbb{R}^3)$ holds, it is easy to check that  
\begin{align*}
\frac{1}{2}\leq1+\varepsilon q ^\varepsilon \leq \frac{3}{2}, 
\end{align*}
which yields
\begin{align}\label{G3.5}
I_3\lesssim&\,  \varepsilon \|q^\varepsilon\|_{L^3} \|b^\varepsilon-u^\varepsilon\|_{L^2} \|u^\varepsilon\|_{L^6}+\varepsilon \|q^\varepsilon\|_{L^6}\|a^\varepsilon\|_{L^6} \|u^\varepsilon\|_{L^3}^2\nonumber\\
\lesssim&\, \varepsilon\sigma \big( \|b^\varepsilon-u^\varepsilon\|_{L^2}^2+ \|\nabla_{x}(a^\varepsilon,u^\varepsilon)\|_{L^2}^2    \big )
\end{align}
Using Lemma \ref{L2.1}, we also have
\begin{align}\label{G3.6}
I_4+I_5\lesssim&\, \|q^\varepsilon\|_{L^3}\|q^\varepsilon\|_{L^6} \|\nabla_{x} u^\varepsilon\|_{L^2}+    \|u^\varepsilon\|_{L^3}\|u^\varepsilon\|_{L^6} \|\nabla_{x} u^\varepsilon\|_{L^2}\nonumber\\
\lesssim&\, \sigma \|\nabla_{x}q^\varepsilon\|_{L^2}^2.
\end{align}
For the remaining terms $I_6$ and $I_7$, using an argument analogous to that in \cite[Lemma 2.1]{DL-KRM-2013}, we get
\begin{align}\label{G3.7}
I_6+I_7\lesssim&\, \|u^\varepsilon\|_{L^3}\|(a^\varepsilon,b^\varepsilon)\|_{L^6} \|\{\mathbf{I}-\mathbf{P}\}f^\varepsilon\|_{\nu} +   \|u^\varepsilon\|_{L^\infty } \|\{\mathbf{I}-\mathbf{P}\}f^\varepsilon\|_{\nu} ^2 \nonumber\\
\lesssim&\, \sigma \big( \|\nabla_{x}(a^\varepsilon,b^\varepsilon)\|_{L^2}^2+\|\{\mathbf{I}-\mathbf{P}\}f^\varepsilon\|_{\nu} ^2   \big).
 \end{align}
Putting the estimates \eqref{G3.4}--\eqref{G3.7} into \eqref{G3.3}, and then using \eqref{G2.2} and the smallness of $\varepsilon$, we obtain \eqref{G3.2}.
\end{proof}

Next, we provide the estimate for $(q^\varepsilon, u^\varepsilon, f^\varepsilon)$ in $\dot{H}^1 \cap \dot{H}^3$.
\begin{lem}\label{L3.2}
For the strong solution $(q^\varepsilon, u^\varepsilon, f^\varepsilon)$ to the compressible Euler--VFP system \eqref{A1}, there exists a positive constant $\lambda_2$ independent of $\varepsilon$, such that
\begin{align}\label{G3.8}
&\frac{{\rm d}}{{\rm d}t}\sum_{1\leq |\alpha|\leq 3} \Big(  \Big\|\frac{ \sqrt{P^\prime(1+\varepsilon q^\varepsilon)}}{1+\varepsilon q^\varepsilon} \partial^\alpha_{x} q^\varepsilon\Big\|_{L^2}^2+\|\partial^\alpha_{x}u^\varepsilon\|_{ L^2}^2+\|\partial^\alpha_{x}f^\varepsilon\|_{L_v^2( L^2)}^2     \Big)\nonumber\\
&\quad +\lambda_2  \sum_{1\leq |\alpha|\leq 3}\|\{\mathbf{I}-\mathbf{P}\}\partial^\alpha_{x} f^\varepsilon\|_{\nu}^2+\lambda_{2}\sum_{1\leq |\alpha|\leq 3}\|\partial^\alpha_{x}(b^\varepsilon-u^\varepsilon)\|_{ L^2}^2  \lesssim  \sigma\|\nabla_{x}(q^\varepsilon,u^\varepsilon,a^\varepsilon,b^\varepsilon)\|_{H^2}^2, 
\end{align}
for any $0\leq t<T$.    
\end{lem}

\begin{proof}
Applying the operator $\partial^\alpha_{x}$ $(1\leq |\alpha|\leq 3)$ to the equations \eqref{A1}$_1$--\eqref{A1}$_3$ yields
\begin{equation} \label{G3.9}
\left\{
\begin{aligned}
& \partial_{t}\partial^\alpha_{x} q^\varepsilon+u^\varepsilon\cdot \nabla_{x} \partial^\alpha_{x}q^\varepsilon+\frac{1+\varepsilon q^\varepsilon}{\varepsilon}\nabla_{x} \cdot \partial^\alpha_{x}u^\varepsilon =-[\partial^\alpha_{x}, u^\varepsilon\cdot\nabla_{x}]q^\varepsilon- [\partial^\alpha_{x
},  q^\varepsilon \nabla_{x
}\cdot] u^\varepsilon,\\
& \partial_{t}\partial^\alpha_{x}u^\varepsilon+ u^\varepsilon\cdot \nabla_{x}\partial^\alpha_{x}u^\varepsilon+ \frac{1}{\varepsilon} \frac{P^\prime(1+\varepsilon q^\varepsilon)}{1+\varepsilon q^\varepsilon} \nabla_{x
}\partial^\alpha_{x
}q^\varepsilon-\partial^\alpha_{x}(b^\varepsilon-u^\varepsilon)\\
&\quad=-[\partial^\alpha_{x},u^\varepsilon\cdot\nabla_{x}]u^\varepsilon-\frac{1}{\varepsilon} \Big[\partial^\alpha_{x}, \frac{P^\prime(1+\varepsilon q^\varepsilon)}{1+\varepsilon q^\varepsilon} \Big] \nabla_{x}q^\varepsilon+\partial^\alpha_{x}\Big( -\frac{\varepsilon q^\varepsilon(b^\varepsilon-u^\varepsilon)+a^\varepsilon u^\varepsilon}{1+\varepsilon q^\varepsilon} \Big),\\
& \partial_t\partial^\alpha_{x} f^\varepsilon+v\cdot \nabla_{x} \partial^\alpha_{x}f^\varepsilon+u^\varepsilon\cdot \nabla_{v} \partial^\alpha_{x}f^\varepsilon -\partial^\alpha_{x} u^\varepsilon\cdot v\sqrt{M} -\mathcal{L}\partial^\alpha_{x} f^\varepsilon\\
&\quad=[-\partial^\alpha_{x},u^\varepsilon\cdot\nabla_{v}]f^\varepsilon+\frac{1}{2}v\cdot \partial^\alpha_{x}(u^\varepsilon f^\varepsilon).
\end{aligned}
\right.
\end{equation}
where we have used the following identity:
\begin{align*}
\frac{1}{\varepsilon}    [\partial^\alpha_{x},(1+\varepsilon q^\varepsilon){\rm div}_x ] u^\varepsilon=[\partial^\alpha_{x
},  q^\varepsilon \nabla_{x
}\cdot ] u^\varepsilon.
\end{align*}

\begin{rem} \label{rem3.1}
Compared with the standard high-order estimate in
\cite[Lemma 2.2]{DL-KRM-2013}, the present equation \eqref{G3.9}$_2$
contains the pressure commutator
\begin{align*}
-\frac{1}{\varepsilon}
\Big[
\partial^\alpha_x,
\frac{P^\prime(1+\varepsilon q^\varepsilon)}
     {1+\varepsilon q^\varepsilon}
\Big]\nabla_x q^\varepsilon ,
\end{align*}
which is not singular. Indeed, since
\begin{align*}
\nabla_x\Big(
\frac{P^\prime(1+\varepsilon q^\varepsilon)}
     {1+\varepsilon q^\varepsilon}
\Big)
&=
\varepsilon
\Big(
\frac{
P^{\prime\prime}(1+\varepsilon q^\varepsilon)(1+\varepsilon q^\varepsilon)
-
P^\prime(1+\varepsilon q^\varepsilon)}
{(1+\varepsilon q^\varepsilon)^2}
\Big)
\nabla_x q^\varepsilon ,
\end{align*}
every spatial derivative of
\(\frac{P^\prime(1+\varepsilon q^\varepsilon)}
        {1+\varepsilon q^\varepsilon}\)
carries one factor \(\varepsilon\). Consequently, by Lemma \ref{L2.2}, we deduce that 
for \(1\le |\alpha|\le 3\),
\begin{align*}
\Big\|
\frac{1}{\varepsilon}
\Big[
\partial^\alpha_x,
\frac{P^\prime(1+\varepsilon q^\varepsilon)}
     {1+\varepsilon q^\varepsilon}
\Big]\nabla_x q^\varepsilon
\Big\|_{L^2}
\leq C
\|\nabla_x q^\varepsilon\|_{H^2}
\|\nabla_x q^\varepsilon\|_{H^{|\alpha|-1}}
\leq C
\|\nabla_x q^\varepsilon\|_{H^2}^2,
\end{align*}
where the constant $C$ is independent of \(\varepsilon\). Therefore, the pressure
commutator is treated as a standard nonlinear term and does not destroy the
uniform-in-$\varepsilon$ high-order energy estimate.
\end{rem}

Multiplying \eqref{G3.9}$_1$--\eqref{G3.9}$_3$ by $
\frac{P^\prime(1+\varepsilon q^\varepsilon)}
     {(1+\varepsilon q^\varepsilon)^2}
\partial^\alpha_x q^\varepsilon $,  $
\partial^\alpha_x u^\varepsilon $,
 $\partial^\alpha_x f^\varepsilon$,
respectively, and summing the resulting equations, we obtain
\begin{align}\label{G3.10}
&\frac{1}{2}\frac{{\rm d}}{{\rm d}t} \Big(  \Big\|\frac{ \sqrt{P^\prime(1+\varepsilon q^\varepsilon)}}{1+\varepsilon q^\varepsilon} \partial^\alpha_{x} q^\varepsilon\Big\|_{L^2}^2+\|\partial^\alpha_{x}u^\varepsilon\|_{ L^2}^2+\|\partial^\alpha_{x}f^\varepsilon\|_{L_v^2( L^2)}^2    \Big) \nonumber\\
&+\int_{\mathbb R^3}\langle \mathcal{L}\{\mathbf{I}-\mathbf{P}\}\partial^\alpha_{x}f^\varepsilon, \partial^\alpha_{x}f^\varepsilon\rangle\,{\rm d}x + \|\partial^\alpha_{x} (b^\varepsilon-u^\varepsilon)\|_{L^2}^2 \nonumber\\
=\,& \frac{1}{2}\int_{\mathbb R^3} \partial_{t} \Big( \frac{P^\prime(1+\varepsilon q^\varepsilon)}{(1+\varepsilon q^\varepsilon)^2}  \Big)  |\partial^\alpha_{x}q^\varepsilon|^2 \,{\rm d}x- \int_{\mathbb R^3} \frac{P^\prime(1+\varepsilon q^\varepsilon)} {(1+\varepsilon q^\varepsilon)^2} \big( [\partial^\alpha_{x}, u^\varepsilon\cdot\nabla_{x}]q^\varepsilon+[\partial^\alpha_{x
},  q^\varepsilon \nabla_{x
}\cdot ] u^\varepsilon  \big) \partial^\alpha_{x}q^\varepsilon\,{\rm d}x \nonumber\\
&+\frac{1}{2}\int_{\mathbb R^3} \nabla_{x
}\cdot \Big( \frac{P^\prime(1+\varepsilon q^\varepsilon)}{(1+\varepsilon q^\varepsilon)^2} u^\varepsilon \Big) |\partial^\alpha_{x}q^\varepsilon|^2\,{\rm d}x+\frac{1}{\varepsilon }\int_{\mathbb R^3} \nabla_{x}\Big ( \frac{P^\prime(1+\varepsilon q^\varepsilon)}{1+\varepsilon q^\varepsilon} \Big)\partial^\alpha_{x} q^\varepsilon \cdot \partial^\alpha_{x} u^\varepsilon \,{\rm d}x\nonumber\\
&-\int_{\mathbb R^3} \Big( [\partial^\alpha_{x},u^\varepsilon\cdot\nabla_{x}]u^\varepsilon+\frac{1}{\varepsilon} \Big[\partial^\alpha_{x}, \frac{P^\prime(1+\varepsilon q^\varepsilon)}{1+\varepsilon q^\varepsilon} \Big] \nabla_{x}q^\varepsilon        \Big)\cdot \partial^\alpha_{x}u^\varepsilon \,{\rm d}x\nonumber\\
&-\int_{\mathbb R^3} \partial^\alpha_{x}\Big( \frac{\varepsilon q^\varepsilon(b^\varepsilon-u^\varepsilon)-a^\varepsilon u^\varepsilon}{1+\varepsilon q^\varepsilon} \Big)\cdot\partial^\alpha_{x} u^\varepsilon\,{\rm d}x
+\frac{1}{2}\int_{\mathbb R^3} |\partial^\alpha_{x} u^\varepsilon|^2{\rm div}_x  u^\varepsilon\,{\rm d}x\nonumber\\
& -\int_{\mathbb{ R}^3}\langle [\partial^\alpha_{x}, u^\varepsilon\cdot\nabla_{v}]f^\varepsilon , f^\varepsilon \rangle \, {\rm d}x-\frac{1}{2} \int_{\mathbb R^3} \langle \partial^\alpha_{x}(u^\varepsilon \cdot v \partial^\alpha_{x}f^\varepsilon), \partial^\alpha_{x}f^\varepsilon \rangle  \,{\rm d}x\nonumber\\
\equiv:\,& \sum_{i=1}^9J_{i}.
\end{align}
For the term $J_1$, recalling \eqref{A1} and applying \eqref{G3.1} along with Lemma \ref{L2.1}, we have
\begin{align}\label{G3.11}
J_1\lesssim&\, \varepsilon   \|q^\varepsilon\|_{H^3} \|\partial_{t}q^\varepsilon\|_{L^\infty} \|\partial^\alpha_{x}q\|_{L^2}^2 \nonumber\\
\lesssim &\, \varepsilon \sigma \Big (  \|u^\varepsilon\|_{L^\infty} \|\nabla_{x}q^\varepsilon\|_{L^\infty}+\frac{1}{\varepsilon} \|\nabla_{x} u^\varepsilon\|_{L^\infty}     \Big) \|\nabla_{x}q^\varepsilon\|_{H^2}^2 \nonumber\\
\lesssim &\, \sigma \|\nabla_{x}q^\varepsilon\|_{H^2}^2.
\end{align}
For the commutator terms $J_2$, $J_5$, and $J_8$, using Lemma \ref{L2.2} and Young’s inequality, we obtain  \begin{align}\label{G3.12}
& J_2+J_5+J_8\nonumber\\
\lesssim\,& \|q^\varepsilon\|_{L^\infty} \big(\|[\partial^\alpha_{x},u^\varepsilon\cdot\nabla_{x}]q^\varepsilon\|_{L^2}+\|[\partial^\alpha_{x},q^\varepsilon{\rm div}_x ]q^\varepsilon\|_{L^2} \big) \|\partial^\alpha_{x}q^\varepsilon\|_{L^2}+\|[\partial^\alpha_{x},u^\varepsilon\cdot\nabla_{x}]u^\varepsilon\|_{L^2} \|u^\varepsilon\|_{L^2}\nonumber\\
&+ \frac{1}{\varepsilon} \Big\| \Big[
\partial^\alpha_x,
\frac{P^\prime(1+\varepsilon q^\varepsilon)}
     {1+\varepsilon q^\varepsilon}
\Big]\nabla_x q^\varepsilon \Big\|_{L^2} \|\partial^\alpha_{x}u^\varepsilon\|_{L^2}+ \|\partial^\alpha_{x}(u^\varepsilon f^\varepsilon)\|_{L_v^2(L^2)} \|\nabla_{v} \partial^\alpha_{x} f^\varepsilon\|_{L_v^2(L^2)} \nonumber\\
\lesssim\,& \|q^\varepsilon\|_{H^3} \|\nabla_{x}(q^\varepsilon,u^\varepsilon)\|_{H^2}^2 \|\nabla q^\varepsilon\|_{H^2}+\|u^\varepsilon\|_{H^3}\|\nabla (q^\varepsilon,u^\varepsilon)\|_{H^2}^2\nonumber\\
&+ \|f^\varepsilon\|_{L_v^2(H^3)} \|\nabla_{x} u^\varepsilon\|_{H^2} \Big( \|\nabla_{x}(a^\varepsilon,b^\varepsilon)\|_{H^2}+\sum_{1\leq|\alpha|\leq 3} \|\{\mathbf{I}-\mathbf{P}\}\partial^\alpha_{x}f^\varepsilon\|_{\nu }  \Big) \nonumber\\
\lesssim\,& \sigma  \Big(\|\nabla_{x}(q^\varepsilon,u^\varepsilon,a^\varepsilon,b^\varepsilon)\|_{H^2}^2+ \sum_{1\leq|\alpha|\leq 3} \|\{\mathbf{I}-\mathbf{P}\}\partial^\alpha_{x}f^\varepsilon\|_{\nu }^2  \Big). 
\end{align} 
It follows from Lemmas \ref{L2.1}--\ref{L2.2}, H\"{o}lder's and Young's inequalities that
\begin{align}\label{G3.13}
&J_3+J_4+J_6+J_7\nonumber\\
\lesssim \,&\varepsilon\|u^\varepsilon\|_{H^3}\|q^\varepsilon\|_{H^3} \|\nabla_{x}q^\varepsilon\|_{H^2}^2+ \frac{1}{\varepsilon} (\varepsilon\|q^\varepsilon\|_{H^3}\|\nabla_{x}u^\varepsilon\|_{H^2}\|\nabla_{x}q^\varepsilon\|_{H^2}) +\|u^\varepsilon\|_{H^3}\|\nabla_{x} u^\varepsilon\|_{H^2}^2  \nonumber\\
&+\varepsilon^2\|\nabla_{x}(b^\varepsilon-u^\varepsilon)\|_{H^2}\|\nabla_{x}q^\varepsilon\|_{H^2}^2 \|\nabla_{x}u^\varepsilon\|_{H^2}+\varepsilon\|\nabla_{x}a^\varepsilon\|_{H^2} \|\nabla_{x}u^\varepsilon\|_{H^2}^2 \|\nabla_{x}q^\varepsilon\|_{H^2}   \nonumber\\
\lesssim\,& \sigma \big( \|\nabla_{x}(q^\varepsilon,u^\varepsilon)\|_{H^2}^2+ \|\nabla_{x}(b^\varepsilon-u^\varepsilon)\|_{H^2}^2 \big).
\end{align}
For the last term $J_9$, applying Lemma \ref{L2.1} and the decomposition \eqref{G2.1} gives rise to 
\begin{align}\label{G3.14}
J_9\lesssim &\, \|\nabla_{x}(u^\varepsilon f^\varepsilon)\|_{L_v^2(L^2)} \Big( \|\nabla_{x}(a^\varepsilon,b^\varepsilon)\|_{H^2}  +\sum_{1\leq|\alpha|\leq 3} \|\{\mathbf{I}-\mathbf{P}\}\partial^\alpha_{x}f^\varepsilon\|_{\nu }\Big) \nonumber\\
\lesssim&\, \sigma \Big(\|\nabla_{x}(a^\varepsilon,b^\varepsilon,u^\varepsilon)\|_{H^2}^2+ \sum_{1\leq|\alpha|\leq 3} \|\{\mathbf{I}-\mathbf{P}\}\partial^\alpha_{x}f^\varepsilon\|_{\nu }^2  \Big).
\end{align}
Therefore, substituting the estimates \eqref{G3.11}--\eqref{G3.14} into \eqref{G3.10}, utilizing \eqref{G2.2} and then summing over    $1\leq |\alpha|\leq 3$,  we obtain \eqref{G3.8}.  This completes the proof of Lemma \ref{L3.2}.
\end{proof}

To absorb $\|\nabla_{x}(a^\varepsilon,b^\varepsilon)\|_{H^2}$ on the right-hand side of \eqref{G3.2} and \eqref{G3.8}, in the spirit of Kawashima’s hyperbolic-parabolic dissipation estimates in \cite{Ks-1983}, we study the following equations of $a^\varepsilon$ and $b^\varepsilon$:
\begin{equation} \label{AB}
\left\{\begin{aligned}
&\partial_{t}a^\varepsilon+{\rm div}_x  b^\varepsilon=0,\\
&\partial_{t} b^\varepsilon_i+\partial_{i} a^\varepsilon+\sum_{j=1}^3\partial_j\Gamma_{ij}(\{\mathbf{I}-\mathbf{P}\}f^\varepsilon)= u^\varepsilon_i-b^\varepsilon_i+u_i^\varepsilon a^\varepsilon,  \\
&\partial_{i}b^\varepsilon_j+\partial_j b^\varepsilon_i- (u^\varepsilon_ib^\varepsilon_j+u^\varepsilon_jb^\varepsilon_i)=-\partial_t \Gamma_{ij}(\{\mathbf{I}-\mathbf{P}\}f^\varepsilon)+\Gamma_{ij}(\ell^\varepsilon+ r^\varepsilon ),  
 \end{aligned}
 \right.
\end{equation}
for $1\leq i,j\leq 3$,
where $\Gamma_{ij}$ is the moment functional defined by
\begin{align*}
\Gamma_{ij}(g)=\langle (v_{i}v_{j}-1)\sqrt{M}, g\rangle,    
\end{align*}
for any $g=g(v)$, and $\ell^\varepsilon$, $r^\varepsilon$ represent  
\begin{align*}
\ell^\varepsilon:=&\, \mathcal{L}\{\mathbf{I}-\mathbf{P}\}f^\varepsilon-v\cdot\nabla_{x}\{\mathbf{I}-\mathbf{P}\}f^\varepsilon,    \\
r^\varepsilon:=&\,  -u^\varepsilon\cdot\nabla_v\{\mathbf{I}-\mathbf{P}\}f^\varepsilon+\frac{1}{2}u^\varepsilon\cdot v\{\mathbf{I}-\mathbf{P}\}f^\varepsilon.
\end{align*}
Similar to \cite{Ww-CMS-2024,DL-KRM-2013,CDM-KRM-2011} , we define the temporal functional $\mathcal{E}^\varepsilon_0(t)$ as
\begin{align}\label{G3.15}
\mathcal{E}^\varepsilon_0(t):=\sum_{|\alpha|\leq 2}\sum_{i,j=1}^3\int_{\mathbb{R}^3}\partial^\alpha_{x}(\partial_ib^\varepsilon_j
+\partial_jb^\varepsilon_i)\partial^{\alpha}_{x}\Gamma_{ij}(\{\mathbf{I}-\mathbf{P}\}f^\varepsilon)\,\mathrm{d}x-\sum_{|\alpha|\leq 2}\int_{\mathbb{R}^3}\partial^\alpha_{x} a^\varepsilon\partial^{\alpha}_{x}{\rm div}_x  b^\varepsilon\,\mathrm{d}x. 
\end{align}

The following lemma can be proved by taking the similar arguments to that in  \cite[Lemma 2.4]{CDM-KRM-2011}. For the sake of brevity, we omit the details here.
\begin{lem} \label{L3.3}
For the strong solution $(q^\varepsilon, u^\varepsilon, f^\varepsilon)$ to the compressible Euler--VFP system \eqref{A1}, there exists a positive constant $\lambda_3$ independent of $\varepsilon$, such that
\begin{align}\label{G3.16}
\frac{\rm d}{{\rm d}t}\mathcal{E}^\varepsilon _{0}(t)+\lambda_3\|\nabla(a^\varepsilon,b^\varepsilon)\|_{H^{2}}^{2}
\lesssim&\, \|\{\mathbf{I}-\mathbf{P}\}f^\varepsilon\|_{L_{v}^{2}(H^{3})}^{2}+\|b^\varepsilon-u^\varepsilon\|_{H^{2}}^{2}, 
\end{align} 
for any $0 \leq t <T$.
\end{lem}

Finally, we give the estimate of $\|\nabla_{x}q^\varepsilon\|_{H^2}$.

\begin{lem} \label{L3.4}
For the strong solution $(q^\varepsilon, u^\varepsilon, f^\varepsilon)$ to the compressible Euler--VFP system \eqref{A1}, there exists a positive constant $\lambda_4$ independent of $\varepsilon$, such that
\begin{align}\label{G3.17}
\varepsilon\frac{\rm d}{{\rm d}t}\sum_{|\alpha|\leq 2}\int_{\mathbb{R}^3}\partial^\alpha_{x} u^\varepsilon\cdot\partial^\alpha_{x}\nabla_{x}q^\varepsilon\,\mathrm{d}x+\lambda_4\|\nabla q^\varepsilon\|_{H^{2}}^{2}
\lesssim&\,\big(\|\nabla_{x} u^\varepsilon\|_{H^2}^2+\|b^\varepsilon-u^\varepsilon\|_{H^{2}}^{2}\big),  
\end{align} 
for any $0 \leq t <T$.
\end{lem}

\begin{proof}
For $|\alpha|\leq 2$, in view of \eqref{A1}$_2$, we have
\begin{align}\label{G3.18}
{P^\prime(1)}\|\nabla_{x}\partial^\alpha_{x}q^\varepsilon\|_{L^2}^2=\, &-{\varepsilon}\int_{\mathbb{R}^3}\nabla_{x}\partial^\alpha_{x}q^\varepsilon\cdot\partial^\alpha_{x}\partial_t u^\varepsilon\,\mathrm{d}x+{\varepsilon}
\int_{\mathbb{R}^3}\nabla_{x}\partial^\alpha_{x}q^\varepsilon\cdot\partial^\alpha_{x}
\Big(\frac{b^\varepsilon-u^\varepsilon}{1+\varepsilon q^\varepsilon}\Big)\,\mathrm{d}x   \nonumber\\
&-{\varepsilon}\int_{\mathbb{R}^3}\nabla\partial^\alpha_{x} q^\varepsilon\cdot\partial^\alpha_{x}\Big(\frac{a^\varepsilon u^\varepsilon}{1+\varepsilon q^\varepsilon}\Big)\,\mathrm{d}x-{\varepsilon}
\int_{\mathbb{R}^3}\nabla_{x}\partial^\alpha_{x}
q^\varepsilon\cdot\partial^\alpha_{x}(u^\varepsilon\cdot\nabla_{x} u^\varepsilon)\,\mathrm{d}x\nonumber\\
&-\int_{\mathbb{R}^3}\nabla_{x}\partial^\alpha_{x
}q^\varepsilon\cdot\partial^\alpha_{x}\Big(\Big(\frac{P^{\prime}(1+\varepsilon q^\varepsilon)}{1+\varepsilon q^\varepsilon}-P^{\prime}(1)       \Big)\nabla_{x}q^\varepsilon\Big)\,\mathrm{d}x \nonumber\\
\equiv:\,&\sum_{j=1}^{5}K_j.
\end{align}
For the term $K_1$, with the help of \eqref{A1}$_1$ and Lemmas \ref{L2.1}--\ref{L2.2}, we arrive at
\begin{align}\label{G3.19}
K_1=&-\varepsilon\frac{{\rm d}}{{\rm d}t} \int_{\mathbb{R}^3}\nabla_{x}  \partial^\alpha_{x}q^\varepsilon\cdot\partial^\alpha_{x} u^\varepsilon \,\mathrm{d}x+\varepsilon \int_{\mathbb{R}^3}\nabla_{x}  \partial^\alpha_{x}\partial_t q^\varepsilon\cdot\partial^\alpha_{x} u^\varepsilon\,\mathrm{d}x\nonumber\\
=&-\varepsilon\frac{{\rm d}}{{\rm d}t} \int_{\mathbb{R}^3}\nabla_{x}  \partial^\alpha_{x}q^\varepsilon\cdot\partial^\alpha_{x} u^\varepsilon \,\mathrm{d}x+\int_{\mathbb{R}^3}\partial^\alpha_{x}\big((1+\varepsilon q^\varepsilon) {\rm div}_x  u^\varepsilon+ \varepsilon\nabla_{x} q^\varepsilon\cdot u^\varepsilon  \big)  \partial_{x}^\alpha {\rm div}_x  u^\varepsilon\,\mathrm{d}x\nonumber\\ 
\leq& -\varepsilon\frac{{\rm d}}{{\rm d}t} \int_{\mathbb{R}^3}\nabla_{x}  \partial^\alpha_{x}q^\varepsilon\cdot\partial^\alpha_{x} u^\varepsilon \,\mathrm{d}x+\big( (1+\varepsilon\|q^\varepsilon\|_{H^2})\|\nabla_{x} u^\varepsilon\|_{H^1}+\varepsilon \|q^\varepsilon\|_{H^3}\|\nabla_{x}u^\varepsilon\|_{H^1}  \big)\|\nabla_{x}u^\varepsilon\|_{H^2} \nonumber\\
\leq& -\varepsilon\frac{{\rm d}}{{\rm d}t} \int_{\mathbb{R}^3}\nabla_{x}  \partial^\alpha_{x}q^\varepsilon\cdot\partial^\alpha_{x} u^\varepsilon \,\mathrm{d}x+C\|\nabla_{x} u^\varepsilon\|_{H^2}^2.
\end{align}
For the remaining terms, it follows directly from H\"{o}lder’s and Young’s inequalities that
\begin{align}\label{G3.20}
K_2\leq&\,   \frac{1}{5} \|\nabla_{x}\partial^\alpha_{x}q^\varepsilon\|_{L^2}^2+C (1+\varepsilon\|q^\varepsilon\|_{H^3})^2 \|b^\varepsilon-u^\varepsilon\|_{H^2}^2\nonumber\\
\leq&\, \frac{1}{5} \|\nabla_{x}\partial^\alpha_{x}q^\varepsilon\|_{L^2}^2+ C\|b^\varepsilon-u^\varepsilon\|_{H^2}^2, \\ \label{G3.21}
K_3\leq &\, \frac{1}{5} \|\nabla_{x}\partial^\alpha_{x}q^\varepsilon\|_{L^2}^2+C(1+\varepsilon\|q^\varepsilon\|_{H^3})^2\|\nabla_{x} u^\varepsilon\|_{H^1}^2 \|a^\varepsilon\|_{H^2}^2\nonumber\\
\leq&\, \frac{1}{5} \|\nabla_{x}\partial^\alpha_{x}q^\varepsilon\|_{L^2}^2+ C \|\nabla_{x} u^\varepsilon\|_{H^1}^2, \\ \label{G3.22}
K_4 \leq&\, \frac{1}{5} \|\nabla_{x}\partial^\alpha_{x}q^\varepsilon\|_{L^2}^2+C\|u^\varepsilon\|_{H^3}^2 \|\nabla_{x}u^\varepsilon\|_{H^2}^2 \nonumber\\
\leq&\, \frac{1}{5} \|\nabla_{x}\partial^\alpha_{x}q^\varepsilon\|_{L^2}^2+ C \|\nabla_{x} u^\varepsilon\|_{H^2}^2, \\ \label{G3.23}
K_5\leq&\, \frac{1}{5} \|\nabla_{x}\partial^\alpha_{x}q^\varepsilon\|_{L^2}^2+C\|q^\varepsilon\|_{H^3}^2 \|\nabla_{x}q^\varepsilon\|_{H^2}^2 \nonumber\\
\leq&\, \frac{1}{5} \|\nabla_{x}\partial^\alpha_{x}q^\varepsilon\|_{L^2}^2+C\sigma \|\nabla_{x}q^\varepsilon\|_{H^2}^2. 
\end{align}
Inserting the estimates \eqref{G3.19}–\eqref{G3.23} into \eqref{G3.18} and summing the resulting inequalities over $|\alpha|\leq 2$, we eventually obtain \eqref{G3.17}.
\end{proof}

\subsection{Proof of uniform global well-posedness}
With Lemmas \ref{L3.1}–\ref{L3.4} in hand, we are now in a position to establish the global a priori estimates for the solution $(q^\varepsilon, u^\varepsilon, f^\varepsilon)$ to the scaled compressible Euler--VFP system \eqref{A1}.
\begin{proof}[Proof of Theorem \ref{Th1}]
We define the temporal energy functional $\mathcal{X}_{\varepsilon}(t)$ and the corresponding dissipation rate $\mathcal{D}_\varepsilon(t)$ as follows:
\begin{align}\label{G3.24}
\mathcal{X}_{\varepsilon}(t):=\,& P^\prime(1)\|q^\varepsilon\|_{L^2}^2+\|u^\varepsilon\|_{H^3}^2+\|f^\varepsilon\|_{L_v^2(H^3)}^2+\sum_{1\leq |\alpha|\leq 3}  \Big\|\frac{ \sqrt{P^\prime(1+\varepsilon q^\varepsilon)}}{1+\varepsilon q^\varepsilon} \partial^\alpha_{x} q^\varepsilon\Big\|_{L^2}^2
\nonumber\\
&+ \tau_1 \mathcal{E}^\varepsilon_0(t
)+\tau_2  \varepsilon \sum_{|\alpha|\leq 2} \int_{\mathbb{ R}^3}\partial^\alpha_{x} u^\varepsilon\cdot\partial^\alpha_{x}\nabla_{x}q^\varepsilon\,{\rm d}x, \\ \label{G3.25}
\mathcal{D}_{\varepsilon}(t):=&\,\|b^\varepsilon-u^\varepsilon\|_{H^3}^2+\|\nabla_{x}(q^\varepsilon, a^\varepsilon,b^\varepsilon )\|_{H^2}^2+\sum_{|\alpha|\leq3}\|\{\mathbf{I}-\mathbf{P}\}\partial^\alpha_{x} f^\varepsilon\|_{\nu}^2,
\end{align}
where   $0<\tau_1, \tau_2 \ll 1$ with $\tau_2\ll\tau_1$ are sufficiently small constants.

Since $\tau_1 > 0$ and $\tau_2 > 0$ are sufficiently small, it follows from the assumption \eqref{G3.1} together with \eqref{G3.15} that,
\begin{align*}
 \mathcal{X}_{\varepsilon}(t)\backsim \|(q^\varepsilon,u^\varepsilon)(t)\|_{H^3}^2+\|f^\varepsilon(t)\|_{L_v^2(H^3)}^2,
\end{align*}
uniformly for $0\leq t<T$. In addition, with suitable choices of the constants $\tau_1$ and $\tau_2$, the sum of the inequalities \eqref{G3.2}, \eqref{G3.8}, $\tau_1\times$ \eqref{G3.16}, and $\tau_2\times$\eqref{G3.17} leads to
\begin{align}\label{G3.26}
\frac{{\rm d}}{{\rm d}t}\mathcal{X}_{\varepsilon}(t)+  \lambda_{5}   \mathcal{D}_{\varepsilon}(t)\leq 0,
\end{align}
for all $0 \leq t < T$, where we have used the fact that 
\begin{align*}
\|\nabla_{x} u^\varepsilon\|_{H^2}\lesssim\|b^\varepsilon-u^\varepsilon\|_{H^3}+ \|\nabla_{x}b^\varepsilon\|_{H^2}.    
\end{align*}
Integrating \eqref{G3.26} with respect to $t$ yields
\begin{align}\label{G3.27}
 \mathcal{X}_{\varepsilon}(t)+\lambda_{5}\int_0^t   \mathcal{D}_{\varepsilon}(s) {\rm d}s\leq  \mathcal{X}_{\varepsilon}(0),
\end{align}
for any $0\leq t< T$.
Besides, \eqref{G3.1} can be justified by choosing
\begin{align*}
  \mathcal{X}_{\varepsilon}(0)\backsim  \|(q_0^\varepsilon, u_0^\varepsilon)\|_{H^3}^2+  \|f_0^\varepsilon\|_{L_v^2(H^3)}^2 
\end{align*}
sufficiently small.

The existence and uniqueness of a local solution to the Cauchy problem
\eqref{A1} can be obtained by the Banach contraction mapping principle; see, for instance, \cite{Gy-CPAM-2002,LMW}. For brevity, we omit the details of the proof here.
Combining the local existence result, the standard continuity argument, and
the {\it uniform a priori estimate} \eqref{G3.27}, we extend the local solution
globally in time and obtain the global existence of strong solutions
\((q^\varepsilon,u^\varepsilon,f^\varepsilon)\) to the scaled compressible Euler-VFP system \eqref{A1}.
Moreover, by the maximum principle established in \cite{GY-iumj-2004},  we deduce that
\begin{align*}
 F^\varepsilon=M+\sqrt M f^\varepsilon\geq 0.   
\end{align*}
Thus, we complete the proof of Theorem \ref{Th1}.
\end{proof}

\section{Estimate of the pressure function for the inompressible Euler-VFP system}
In this section, based on the estimate \eqref{TGG2}, we will prove Proposition \ref{prop3} below. 
\begin{proof}[Proof of Proposition \ref{prop3}]
We first introduce the orthogonal projectors $\mathbb Q$ and $\mathbb P$, defined by
\begin{align*}
\mathbb Q:=-\nabla_{x}(-\Delta_{x})^{-1} {\rm div}_{x}, \quad \mathbb P:= {\rm Id}-\mathbb Q.
\end{align*}
Applying the operator $\mathbb Q$ to \eqref{A2}$_1$ and using \eqref{A2}$_2$, we obtain 
\begin{align}\label{NJKG4.1}
\nabla_{x}\pi=\mathbb Q(b-u-au-u\cdot\nabla_{x} u),    
\end{align}
which, together with the boundedness property of $\mathbb Q$, gives
\begin{align}\label{G4.1}
\|\nabla_{x}\pi\|_{H^2}\lesssim&\, \|b-u\|_{H^2}+ \|\nabla_{x}a\|_{H^1}\|\nabla_{x}u\|_{H^1}+\|\nabla_{x}u\|_{H^1}\|\nabla^2_{x} u\|_{H^1}\nonumber\\
\lesssim &\, \|b-u\|_{H^2}+\varepsilon_1^{\frac{1}{2}} \|\nabla_{x}u\|_{H^2}\lesssim \varepsilon_1^{\frac{1}{2}}.
\end{align}
Thus, we complete the proof of \eqref{TGGG1}.

Next, we prove \eqref{TGGG2}. On the one hand, from   \eqref{G4.1}, it holds that
\begin{align}\label{G4.2}
\int_0^\infty \|\nabla_{x}\pi\|^2_{H^2}\,{\rm d}\tau    \lesssim&\, \int_0^\infty \big( \|b-u\|_{H^2}^2+\varepsilon_1\|\nabla_{x} u\|_{H^2}^2\big)\,{\rm d}\tau\nonumber\\
\lesssim&\, \int_0^\infty \big( \|b-u\|_{H^3}^2+\varepsilon_1\|\nabla_{x} b\|_{H^2}^2\big)\,{\rm d}\tau\lesssim \varepsilon_1.
\end{align}
On the other hand, by  denoting
\begin{align*}
K:=b-u-au-u\cdot\nabla_{x}u.
\end{align*}
and noticing 
\begin{align*}
\nabla_{x}D_{t}\pi=&\,  D_{t} \nabla_{x}\pi+ (\nabla_{x} u)^\top\cdot \nabla_{x}\pi \nonumber\\
=&\, D_{t}\mathbb QK+ (\nabla_{x} u)^\top\cdot \nabla_{x}\pi\nonumber\\
=&\, \mathbb{Q}D_tK+[u\cdot\nabla_{x},\mathbb Q]K+ (\nabla_{x} u)^\top\cdot \nabla_{x}\pi,
\end{align*}
we have
\begin{align}\label{G4.3}
\|\nabla_{x}D_{t}\pi\|_{H^1}\lesssim&\,   \|\mathbb{Q}D_tK\|_{H^1}+\|[u\cdot\nabla_{x},\mathbb Q]K\|_{H^1}+ \|(\nabla_{x} u)^\top\cdot \nabla_{x}\pi\|_{H^1}\nonumber\\  
\lesssim&\, \|\mathbb{Q}D_tK\|_{H^1}+\big( \|\nabla_{x}\pi\|_{H^2}+\|K\|_{H^1}
\big) \|\nabla_{x} u\|_{H^2}\nonumber\\
\lesssim&\, \|\mathbb{Q}D_tK\|_{H^1}+\varepsilon_1^\frac{1}{2} \|\nabla_{x} u\|_{H^2},
\end{align}
where we have used \eqref{G4.1}.

We now estimate the remaining term $\|\mathbb Q D_{t} K\|_{H^1}$ in \eqref{G4.3}.
From \eqref{A2}$_1$--\eqref{A2}$_3$, we compute that
\begin{align*}
D_t u=&\,-\nabla_{x}\pi +b-u-au,   \\
D_{t}a=&\, -{\rm div}_x  b+u\cdot\nabla_{x}a=-\nabla_{x}(b-u)+u\cdot\nabla_{x}a,\\
D_{t}(b-u)=&\, -\nabla_{x}a-{\rm div}_x  \Gamma({\{\mathbf I-\mathbf P\}f})-2(b-u)+u\cdot \nabla_{x}b+\nabla_{x}\pi+2a u.
\end{align*}
Moreover, we have
\begin{align*}
D_{t}K=&\,   D_{t}(b-u)-D_{t}(au)-D_{t}(u\cdot\nabla_{x}u)\nonumber\\
=&\,D_{t}(b-u)-(D_{t}a) u-a D_{t}u-(D_{t}u)\cdot\nabla_{x} u-u\cdot \nabla_{x}(D_{t} u) +(u\cdot \nabla_{x} u)\cdot\nabla_{x} u.
\end{align*}
Through direct calculations, we arrive at  
\begin{align}\label{G4.4}
\|\mathbb QD_{t}K\|_{H^1}\lesssim&\, \|D_{t}K\|_{H^1}\nonumber\\
  \lesssim&\, \big(1+\varepsilon_1^\frac{1}{2}\big)\Big(\|\nabla_{x}(a,b,u,\pi)\|_{H^2}+\|b-u\|_{H^3}+\sum_{1\leq
\alpha\leq 2}\|\{\mathbf{I}-
\mathbf{P}\}\partial^\alpha_{x}f\|_{\nu}  \Big).
\end{align}
Putting the estimate \eqref{G4.4} into \eqref{G4.3} yields
\begin{align}\label{G4.5}
\int_0^\infty \|\nabla_{x}D_{t}\pi\|_{H^1}^2\,{\rm d}\tau\lesssim   (1+\varepsilon_1) \mathcal{X}_0\lesssim \varepsilon_1,
\end{align}
where we have used \eqref{TGG1} and \eqref{TGG2}.
Combining \eqref{G4.2} and \eqref{G4.5}, we obtain \eqref{TGGG2}.

Finally, we show  \eqref{TGGG3} and \eqref{TGGG4}. Utilizing Lemma \ref{L2.3}, we have
\begin{align*}
\|D_{t}\pi\|_{\dot H^{-1}}\lesssim &\, \|\partial_{t}\pi\|_{\dot H^{-1}}+  \|u\cdot\nabla_{x}\pi\|_{\dot H^{-1}}\nonumber\\
\lesssim &\,  \|\partial_{t}\pi\|_{\dot H^{-1}}+  \|u\cdot\nabla_{x}\pi\|_{L^{\frac{6}{5}}}\nonumber\\
\lesssim&\,  \|\partial_{t}\pi\|_{\dot H^{-1}}+\|u\|_{L^3}\|\nabla_{x}\pi\|_{L^2},
\end{align*}
which, together with \eqref{TGGG2}, yields
\begin{align}\label{G4.6}
\int_0^t \|D_{t}\pi\|_{\dot H^{-1}}^2\,{\rm d}\tau\lesssim&\, \int_0^t \|\partial_{t}\pi\|_{\dot H^{-1}}^2\,{\rm d}\tau+\varepsilon_1\int_0^t \|\nabla_{x
}\pi\|_{L^2}^2\,{\rm d}\tau\lesssim \int_0^t \|\partial_{t}\pi\|_{\dot H^{-1}}^2\,{\rm d}\tau+\varepsilon_1 .
\end{align}
It remains to estimate $\partial_{t}\pi$ in \eqref{G4.6}. Thanks to \eqref{NJKG4.1} and the divergence-free condition \eqref{A2}$_2$, we have
\begin{align}\label{G4.8}
\|\partial_{t}\pi\|_{\dot H^{-1}}\lesssim&\, \|\mathbb Q\big(\partial_{t}(b-u)-\partial_{t}(au)-\partial_{t}(u\cdot\nabla_{x}u)\big)\|_{\dot H^{-1}}\nonumber\\
\lesssim&\, \|\partial_{t} (b-u)\| _{\dot H^{-1}}+\|(\partial_{t}a) u \|_{\dot H^{-1}}+\|a(\partial_{t} u)\|_{\dot H^{-1}}+\|\partial_{t} u\cdot\nabla_{x} u\|_{\dot H^{-1}}+\|u\cdot\nabla_{x}\partial_{t}u\|_{\dot H^{-1}}\nonumber\\
\lesssim&\,\|\partial_{t} (b-u)\| _{\dot H^{-1}}+\|\partial_{t}(a,u)\|_{L^2} \|(a,u) \|_{ L^3}+\|\partial_{t} u\|_{L^2}\|\nabla_{x} u\|_{L^3}+\|u\|_{L^2}\|\partial_{t}u\|_{ L^2} .
\end{align}
Then, to handle the estimate of $\|\partial_{t}(b - u)\|_{\dot H^{-1}}$ in \eqref{G4.8}, we first study the estimate of $\|b - u\|_{\dot H^{-1}}$. To begin with, we write the equation of $b - u$ as  
\begin{align}\label{G4.9}
\partial_{t}(b-u) +2(b-u)=2au+u\cdot\nabla_{x} u+\nabla_{x}\pi -\nabla_{x}a-{\nabla_{x}}\cdot\Gamma (\{\mathbf{I}-\mathbf{P}\}f).
\end{align}
To overcome the difficulty associated with $\nabla_{x}\pi$, we apply the operator $\mathbb P$ to \eqref{G4.9}, and then obtain
\begin{align}\label{G4.10}
\partial_{t}(\mathbb P b-u)+2(\mathbb Pb-u)= \mathbb P\big( 
2au+u\cdot\nabla_{x} u-\nabla_{x}a-{\rm div}_x \Gamma(\{\mathbf{I}-\mathbf{P}\}f)\big). 
\end{align}
On the other hand, letting the operator $\mathbb Q$ act on the equation of $b$ gives  
\begin{align}\label{G4.11}
\partial_{t}\mathbb Q b+\mathbb Q b=\mathbb Q\big(\nabla_{x}a-{\rm div}_x \Gamma (\{\mathbf{I}-\mathbf{P}\}f) +au\big).    
\end{align}

To capture the damping structure of the equations \eqref{G4.10} and \eqref{G4.11}, we rewrite them as:
\begin{align}\label{G4.12}
  (\mathbb Pb-u)= e^{-2t} (\mathbb Pb_0-u_0)+  \int_0^t e^{-2(t-\tau)}\mathbb P\big( 
2au+u\cdot\nabla_{x} u-\nabla_{x}a-{\rm div}_x \Gamma(\{\mathbf{I}-\mathbf{P}\}f)\big)\, {\rm d}\tau,
\end{align}
and
\begin{align}\label{G4.13}
   \mathbb Qb = e^{-t} \mathbb Qb_0 +  \int_0^t e^{-(t-\tau)}\mathbb Q\big( 
2au-\nabla_{x}a-{\rm div}_x \Gamma(\{\mathbf{I}-\mathbf{P}\}f)\big)\, {\rm d}\tau.
\end{align}
Then, by Lemmas \ref{L2.3} and \ref{L2.4} and the fact that $b - u = (\mathbb{P}b - u)+\mathbb{Q}b$, we have
\begin{align}\label{G4.14}
\|  b-u\|_{\dot H^{-1}}\lesssim &\,  e^{-t} \|(b_0,u_0)\|_{L^{\frac{6}{5}}} +\int_0^te^{-(t-\tau)}\big(\|a u\|_{L^\frac{6}{5}}+ \|u\cdot\nabla_{x}u\|_{L^{\frac{6}{5}}}+\|\nabla_{x}a\|_{\dot H^{-1}} \big)\,{\rm d}\tau\nonumber\\
&+\int_0^t e^{-(t-\tau)}\|{\rm div}_x \Gamma(\{\mathbf{I}-\mathbf{P}\}f)\|_{\dot H^{-1}}\,{\rm d}\tau \nonumber\\
\lesssim&\, e^{-t}\|(b_0,u_0)\|_{L^1\cap L^2}+ \int_0^t e^{-(t-\tau)} \big((\|a\|_{L^2}+\|\nabla_{x} u\|_{L^2})(\|u\|_{L^3}+1)+\|f\|_{L_v^2( L^2)}   \big)\,{\rm d}\tau \nonumber\\
\lesssim &\, \varepsilon_1^{\frac{1}{2}}(1+t)^{-\frac{3}{4}+\sigma},
\end{align}
where we have used the decay estimate \eqref{TGG3}.
From \eqref{G4.10}, \eqref{G4.11} and \eqref{G4.14}, we get
\begin{align}\label{G4.15}
\|\partial_{t}(b-u)\|_{\dot H^{-1}}\lesssim&\, \|a u\|_{L^{\frac{6}{5}}}+\|u\cdot \nabla_{x} u \|_{L^\frac{6}{5}}+\|\nabla_{x} a\|_{\dot H^{-1}} +\|{\rm div}_x \Gamma(\{\mathbf{I}-\mathbf{P}\}f)\|_{\dot H^{-1}}  +\|b-u\|_{\dot H^{-1}}\nonumber\\
\lesssim&\, \|a\|_{L^3}\|u\|_{L^2}+\|u\|_{L^3}\|\nabla_{x} u\|_{L^2}+\|a\|_{L^2}+\|f\|_{L_v^2(L^2)}+\|b-u\|_{\dot H^{-1}}\nonumber\\
\lesssim&\, \varepsilon_1^{\frac{1}{2}}(1+t)^{-\frac{3}{4}+\sigma}.
\end{align}

Through direct calculations, we arrive at 
\begin{align}\label{G4.16}
\|\partial_{t} u\|_{L^2}=\|\partial_{t} \mathbb Pu\|_{L^2 }\lesssim  \|u\cdot\nabla_{x}u\|_{L^2}+\|au\|_{L^2}+\|b-u\|_{L^2} 
\lesssim   \varepsilon_1^{\frac{1}{2}}(1+t)^{-\frac{3}{4}+\sigma},
\end{align}
 and 
 \begin{align}\label{G4.17}
 \|\partial_{t}a\|_{L^2}\lesssim \|\nabla_{x} b\|_{L^2} \lesssim \varepsilon_1^{\frac{1}{2}}(1+t)^{-\frac{3}{4}+\sigma}.   
 \end{align}
Putting  the estimates \eqref{G4.15}--\eqref{G4.17} into
\eqref{G4.8} gives
\begin{align*}
\|\partial_{t}\pi\|_{\dot H^{-1}} \lesssim \varepsilon_1^{\frac{1}{2}}(1+t)^{-\frac{3}{4}+\sigma},   
\end{align*}
which together with \eqref{G4.6}, yields
\begin{align*}
\int_0^t \|D_{t}\pi\|_{\dot H^{-1}}^2\,{\rm d}\tau\lesssim \varepsilon_1+\varepsilon_1\int_0^t (1+\tau)^{-\frac{3}{2}+2\sigma} \,{\rm d}\tau\lesssim\varepsilon_1.
\end{align*}
Thus, \eqref{TGGG3} holds.

With the help of \eqref{NJKG4.1} and \eqref{G4.14}, we have
\begin{align*}
\|\pi\|_{L^2}  \lesssim&\, \|b-u\|_{\dot H^{-1}}+\|au\|_{L^{\frac{6}{5}}}+\|u\otimes u\|_{L^2} \nonumber\\
\lesssim&\, \varepsilon_1^\frac{1}{2}+\|a\|_{L^3}\|u\|_{L^2}+\|u\|_{L^3}\|u\|_{L^6}\lesssim\varepsilon_1^\frac{1}{2},
\end{align*}
which implies \eqref{TGGG4}. Thus, we complete the proof of Proposition \ref{prop3}.
\end{proof}

\section{Low Mach number limit of the   compressible Euler--VFP system}
In this section, we justify the global-in-time low Mach number limit for the scaled compressible Euler--VFP system \eqref{A1}. The whole  proof is divided into two steps. First, we establish the global \(H^2\)-error estimate in Theorem \ref{Th4}, which is crucial for the proof as it ensures the strong convergence of the acoustic variables and the full velocity field. 
Subsequently, we utilize this error estimate, along with the uniform-in-\(\varepsilon\) estimate \eqref{TG2} to prove the convergence result in Theorem \ref{Th5}.

\subsection{Proof of Theorem \ref{Th4}}
This subsection is devoted to proving Theorem \ref{Th4}. We now define the differences as follows:
\begin{align*}
\delta q:=q^\varepsilon-\varepsilon [P^\prime(1)]^{-1}\pi,\quad \delta u:=u^\varepsilon-u, ,\quad \delta f:=f^\varepsilon-f,     
\end{align*}
and 
\begin{align*}
\delta a:=a^\varepsilon-a,\quad \delta b:=b^\varepsilon-b,\quad \delta(b-u):=\de b-\de u.
\end{align*}
We also define the error functional $\delta\mathcal{X}(t)$ as
\begin{align}\label{G5.1}
 \delta\mathcal{X}(t):=&\,\sup_{\tau\in [0,t]}\big \{\|(\delta q,\delta u)(\tau)\|_{H^2}^2 +  \|\delta f(\tau)\|_{L_v^2(H^2)}^2\big\}+\int_0^t  \sum_{|\alpha|\leq 2} \|\{\mathbf{I}-\mathbf{P}\}\delta f(\tau)\|_{\nu}^2\,{\rm d}\tau \nonumber\\
 &+\int_0^t \big(\|\nabla_{x}(\delta a,\delta b,\delta q)(\tau)\|_{H^1}^2+ \|\delta (b-u)(\tau)\|_{H^2}^2    \big)\,{\rm d}\tau.
\end{align}
Owing to 
\begin{align*}
\|\nabla_{x} \delta u\|_{H^1}\lesssim \|\delta b-\delta u\|_{H^2}+\|\nabla_{x} \delta b \|_{H^1},    
\end{align*}
it holds that
\begin{align*}
\int_0^t \|\nabla_{x} \delta u (\tau)\|_{H^1}^2\,{\rm d}\tau \lesssim  \delta\mathcal{X}(t).  
\end{align*}

From the systems \eqref{A1} and \eqref{A2}, we  obtain that $(\delta q,\delta u,\delta f)$ satisfies the following error system:
\begin{equation} \label{G5.2}
\left\{\begin{aligned}
&\partial_{t}\delta q+  u^\varepsilon \cdot \nabla_{x} \delta q+\frac{1+\varepsilon q^\varepsilon}{\varepsilon} {\rm div}_x  \delta u= \delta F_1,\\
& \partial_{t}\delta u+u^\varepsilon\cdot \nabla_{x}\delta u+\frac{1}{\varepsilon}\frac{P^\prime(1+\varepsilon q^\varepsilon)}{1+\varepsilon q^\varepsilon}\nabla_{x}\delta q=\delta (b-u)-a(\delta u)-(\delta a)u^\varepsilon+\delta F_{2}+\delta F_3,\\
&\partial_{t}\delta f+v\cdot\nabla_{x}\delta f+u^\varepsilon\cdot \nabla_{v}\delta f-\frac{1}{2}u^\varepsilon\cdot v(\delta f)-\delta u\cdot v\sqrt{M}=\mathcal{L}\delta f-\delta u\cdot\nabla_{v}f+\frac{1}{2}\delta u\cdot v f,
 \end{aligned}
 \right.
\end{equation}
where the definitions of $\delta F_{i}$ with $i=1,2,3,$ are
\begin{equation}\label{ff1}
\left\{\begin{aligned}
&\delta F_1:=-\varepsilon [P^\prime(1)]^{-1} (D_{t}\pi+\delta u\cdot \nabla_{x}\pi),\\
& \delta F_{2}:=-\delta u\cdot \nabla_{x} u -\frac{\varepsilon q^\varepsilon}{1+\varepsilon q^\varepsilon}(b-u-au)  +\frac{\varepsilon q^\varepsilon}{1+\varepsilon q^\varepsilon} \big( a (\delta u)+(\delta a) u^\varepsilon\big)-\frac{\varepsilon q^\varepsilon}{1+\varepsilon q^\varepsilon}\delta (b-u),\\
&\delta F_3:= -\Big([P^\prime(1)]^{-1}\frac{P^\prime(1+\varepsilon q^\varepsilon)}{(1+\varepsilon q^\varepsilon)} -1\Big)\nabla_{x}\pi.
 \end{aligned}
 \right.
\end{equation}

We first estimate   $(\delta q,\delta u,\delta f)$ in the $L^2$-norm.
\begin{lem}\label{L5.1}
It holds that
\begin{align}\label{G5.3}
 &\sup_{0\leq \tau\leq t}  \big ( P^\prime(1) \|\delta q(\tau)\|_{L^2}^2+ \|\delta u(\tau)\|_{L^2}^2+ \|\delta f(\tau)\|_{L_v^2(L^2)}^2 \big)\nonumber\\
&\quad  +\int_0^t\big(\|\{\mathbf{I}-\mathbf{P}\}\delta f(\tau)\|_{\nu}^2
 + \|\delta(b-u)(\tau)\|_{L^2}^2\big
 )\,{\rm d}\tau
\leq \delta \mathcal{X }(0)+C\varepsilon^2+C\big(\varepsilon_0^{\frac{1}{2}}+\varepsilon_1^{\frac{1}{2}}\big) \de\mathcal{X}(t),
\end{align}
where $C > 0 $ is a constant independent of $\varepsilon$ and time $t$.
\end{lem}

\begin{proof}
It follows from \eqref{G2.3} and \eqref{G5.2}$_1$--\eqref{G5.2}$_3$ that 
\begin{align}\label{G5.4}
&\frac{1}{2} \big ( P^\prime(1) \|\delta q(t)\|_{L^2}^2+\|\delta u(t)\|_{L^2}^2+\|\delta f(t)\|_{L_v^2(L^2)}^2       \big) +\int_0^t\big(\|\{\mathbf{I}-\mathbf{P}\}\delta f(\tau)\|_{\nu}^2
 + \|\delta(b-u)(\tau)\|_{L^2}^2\big
 )\,{\rm d}\tau\nonumber\\
\leq\,&\frac{1}{2} \big ( P^\prime(1) \|\delta q(0)\|_{L^2}^2+\|\delta u(0)\|_{L^2}^2+\|\delta f(0)\|_{L_v^2(L^2)}^2       \big)+\frac{1}{2}\int_0^t\!\! \int_{\mathbb R^3}P^\prime(1)|\delta q|^2{\rm div}_x  u^\varepsilon\,{\rm d}x {\rm d}\tau \nonumber\\
&-\int_0^t\!\!\int_{\mathbb R^3} P^\prime(1) q^\varepsilon {\rm div}_x \delta u(\delta q)\, {\rm d}x{\rm d}\tau+\frac{1}{2}\int_0^t\!\! \int_{\mathbb R^3}|\delta u|^2{\rm div}_x  u^\varepsilon\,{\rm d}x {\rm d}\tau \nonumber\\
&-\frac{1}{\varepsilon}\int_0^t\!\!\int_{\mathbb R^3}\Big( \frac{P^\prime(1+\varepsilon q^\varepsilon)}{1+\varepsilon q^\varepsilon}-P^\prime(1)    \Big)\nabla_{x}\delta q\cdot \delta u \,{\rm d}x{\rm d}\tau
+\frac{1}{2}\int_0^t\!\!\int_{\mathbb R^3} u^\varepsilon\cdot \langle v \delta f,\delta f\rangle \, {\rm d}x{\rm d}\tau \nonumber\\
&-\int_0^t \!\!\int_{\mathbb R^3}\langle u^\varepsilon\cdot\nabla_{v}\delta f,\delta f\rangle \,{\rm d}x{\rm d}\tau-\int_0^t \!\!\int_{\mathbb R^3}(\delta a) u^\varepsilon\cdot\delta u\,{\rm d}x{\rm d}\tau+\frac{1}{2}\int_0^t\!\!\int_{\mathbb R^3} \delta u \cdot \langle v  f,\delta f\rangle \, {\rm d}x{\rm d}\tau\nonumber\\
&-\int_0^t \!\!\int_{\mathbb R^3}\langle \delta u \cdot\nabla_{v}  f,\delta f\rangle \,{\rm d}x{\rm d}\tau-\int_0^t \!\!\int_{\mathbb R^3}a (\delta u)\cdot\delta u\,{\rm d}x{\rm d}\tau+\int_0^t \!\!\int_{\mathbb R^3} (\delta q) \delta F_1\,{\rm d}x{\rm d}\tau\nonumber\\\
&+\int_0^t \!\!\int_{\mathbb R^3} (\delta u)\cdot \delta F_2\,{\rm d}x{\rm d}\tau+\int_0^t \!\!\int_{\mathbb R^3} (\delta u)\cdot \delta F_3\,{\rm d}x{\rm d}\tau\nonumber\\
\equiv:\,&  \frac{1}{2} \big ( P^\prime(1) \|\delta q(0)\|_{L^2}^2+\|\delta u(0)\|_{L^2}^2+\|\delta f(0)\|_{L_v^2(L^2)}^2       \big)+\sum_{j=1}^{13}\delta I_{j}.
\end{align}
For the terms $\delta I_1,\dots,\delta I_4$, by virtue of Lemma \ref{L2.1}, H\"{o}lder's and Young's inequalities, we get
\begin{align}\label{G5.5}
\sum^4_{i=1}\delta I_i \lesssim \,& \int_0^t (\|\delta q\|_{L^3}\|\delta q\|_{L^6}\|\nabla_{x}u^\varepsilon \|_{L^2}+\|q^\varepsilon\|_{L^3}\|\nabla_{x}\delta u\|_{L^2}\|\delta q\|_{L^6})\,{\rm d}\tau  \nonumber\\
&+ \int_0^t (\|\delta u\|_{L^3}\|\delta u\|_{L^6}\|\nabla_{x}u^\varepsilon \|_{L^2}+\|q^\varepsilon\|_{L^3}\|\nabla_{x}\delta q\|_{L^2}\|\delta u\|_{L^6})\,{\rm d}\tau  \nonumber\\
\lesssim&\,  \|(\delta q,\delta u)\|_{L_t^\infty(H^1)} \|\nabla_{x}(\delta q,\delta u)\|_{L_t^2(L^2)} \|\nabla_{x} u^\varepsilon\|_{L^2_t(L^2)} \nonumber\\
&+\|q^\varepsilon\|_{L_t^\infty(H^1)} \|\nabla_{x}(\delta q,\delta u)\|_{L_{t}^2(L^2)}^2 \nonumber\\
\lesssim\,& \varepsilon_0^{\frac{1}{2}}\delta\mathcal{X}(t).
\end{align}
For the terms $\delta I_5$, $\delta I_6$, and $\delta I_7$, applying the decomposition \eqref{G2.1} and the property that  {$|v|^k\sqrt{M}\lesssim1$ for any $k\geq 0$}, we have
\begin{align}\label{G5.6}
\delta I_5+\delta I_6+\delta  I_7\lesssim &\, \int_0^t (\|u^\varepsilon\|_{L^\infty}\|\{\mathbf{I}-\mathbf{P}\}\delta f\|_{\nu}^2+ \|u^\varepsilon\|_{L^3} \|\delta a\|_{L^6}\|\delta (b-u)\|_{L^2})\, {\rm d}\tau \nonumber\\
&+\int_0^t \|u^\varepsilon\|_{L^3}\|(\delta a,\delta b)\|_{L^6}\|\{\mathbf{I}-\mathbf{P}\}\delta f\|_{\nu} \,{\rm d}\tau \nonumber\\
\lesssim &\, \|u^\varepsilon\|_{L_t^\infty(H^2)}\int_0^t \big(\|\delta (b-u)\|_{L^2}^2+ \|\nabla_{x}(\delta a,\delta b)\|_{L^2}^2+    \|\{\mathbf{I}-\mathbf{P}\}\delta f\|_{\nu}^2      \big)\,{\rm d}\tau\nonumber\\
\lesssim&\, \varepsilon_0^{\frac{1}{2}} \delta\mathcal{X}(t),
\end{align}
where we  have used the facts that
\begin{align*}
\frac{1}{2} u^\varepsilon \cdot v\mathbf{ P}\delta f- u^\varepsilon \cdot \nabla_{v} \mathbf{P}\delta f=\big((\delta a) u^\varepsilon\cdot v-u^\varepsilon\cdot \delta b+u^\varepsilon\cdot v(\delta b)\cdot v      \big)\sqrt{M},
\end{align*}
and
\begin{align*}
\int_{\mathbb R^3}    \big(u^\varepsilon \cdot v(\delta b)\cdot v\big)\mathbf{P}\delta f \sqrt{M}\,{\rm d}\tau=(\delta a) u^\varepsilon \cdot \delta b.
\end{align*}
Similar to \eqref{G5.6}, we infer that  
\begin{align}\label{G5.7}
\delta I_8+\delta I_9+\delta I_{10}\lesssim  \varepsilon^\frac{1}{2}_1\delta \mathcal{X}(t).   
\end{align}

For the difficult term $\delta I_{11}$ related to the pressure $\nabla_{x}\pi$, which affects the convergence rate of the low Mach number limit, applying Proposition \ref{prop3} and Young's inequality, we find that
\begin{align}\label{G5.8}
\delta I_{11}\lesssim&\, \varepsilon   \|D_{t}\pi\|_{L_t^2(\dot H^{-1})}  \|\delta q\|_{L_t^2(\dot H^1)}+\varepsilon \|\delta u\|_{L_t^\infty(L^3)}\|\delta q\|_{L_t^2(L^6)} \|\nabla_{x}\pi\|_{L_t^2(L^2)}\nonumber\\
\lesssim&\, \varepsilon^2+ \varepsilon_1 \delta\mathcal{X}(t)+\varepsilon\varepsilon_1^{\frac{1}{2}} \delta\mathcal{X}(t)\nonumber\\
\lesssim&\, \varepsilon^2+ \varepsilon_1^{\frac{1}{2}}\delta\mathcal{X}(t).
\end{align}
For the remaining terms $\delta I_{12}$ and $\delta I_{13}$, by Proposition \ref{prop3}, Lemma \ref{L2.1} and Young's inequality, we have  
\begin{align}\label{G5.9}
& \delta I_{12}+\delta I_{13}\nonumber\\
\lesssim \,& \int_0^t (\|\delta u\|_{L^3}\|\delta u\|_{L^6}\|\nabla_{x} u\|_{L^2}+\|b-u\|_{L^2}\|\varepsilon q^\varepsilon\|_{L^3}\|\delta u\|_{L^6}+\|\varepsilon q^\varepsilon\|_{L^2}\|a\|_{L^6}\|u\|_{L^6}\|\delta u\|_{L^6})\,{\rm d}\tau\nonumber\\
&+\int_0^t(\|a\|_{L^6}\|\delta u\|_{L^6}^2\|\varepsilon q^\varepsilon\|_{L^2}+\|\delta a\|_{L^6}\|\delta u\|_{L^6}\|u^\varepsilon\|_{L^6}\|\varepsilon q^\varepsilon\|_{L^2})\,{\rm d}\tau\nonumber\\
&+\int_0^t (\|\delta u\|_{L^6}\|\nabla_{x}\pi\|_{L^2}\|\varepsilon q^\varepsilon\|_{L^3}+ \|\varepsilon q^\varepsilon\|_{L^3}\|\delta(b-u)\|_{L^2}\|\delta u\|_{L^6})\,{\rm d}\tau\nonumber\\
\lesssim\, &\big(1+\varepsilon_0^\frac{1}{2}+\varepsilon_1^\frac{1}{2}\big) \big(\|\nabla_x(\delta u,\delta a)\|_{L_t^2(L^2)}^2+\|\delta(b-u)\|_{L_t^2(L^2)}^2\big)+ \varepsilon \varepsilon_0^\frac{1}{2} \|b-u\|_{L_t^2(L^2)}\|\nabla_{x}\delta u\|_{L_t^2(L^2)}\nonumber\\
&+\varepsilon \varepsilon_0^\frac{1}{2}\|\nabla_{x}(a,u)\|_{L_t^2(L^2)}^2\|\nabla_{x}\delta u\|_{L_t^2(L^2)}+\varepsilon \varepsilon_0^\frac{1}{2}\|\nabla_{x}\pi\|_{L_t^2(L^2)}\|\nabla_{x}\delta u\|_{L_t^2(L^2)}\nonumber\\
\lesssim\,& \big(1+\varepsilon_0^\frac{1}{2}+\varepsilon_1^\frac{1}{2}\big)\delta\mathcal{X}(t)+\varepsilon^2+\varepsilon_0 \delta\mathcal{X}(t) \nonumber\\
\lesssim\,&  \big(1+\varepsilon_0^\frac{1}{2}+\varepsilon_1^\frac{1}{2}\big)\delta\mathcal{X}(t)+\varepsilon^2.
\end{align}
Thus, inserting the estimates \eqref{G5.5}--\eqref{G5.9} into
\eqref{G5.4}, we obtain the desired estimate \eqref{G5.3}.
\end{proof}

Furthermore, we give the estimate of $(\delta q,\delta u,\delta f)$ in the $\dot H^1\cap \dot H^2$-norm.
\begin{lem}\label{L5.2}
It holds that
\begin{align}\label{G5.10}
 &\sup_{0\leq \tau\leq t} \Big(  \sum_{1\leq |\alpha|\leq 
 2}\Big\|\frac{ \sqrt{P^\prime(1+\varepsilon q^\varepsilon)}}{1+\varepsilon q^\varepsilon} \partial^\alpha_{x} \delta q(\tau) \Big\|_{L^2}^2+\| \delta u(\tau)\|_{\dot H^1\cap\dot H^2}^2+\| \delta f (\tau)\|_{L_v^2( \dot H^1\cap\dot H^2)}^2     \Big)  \nonumber\\
 &\quad +\int_0^t \Big(\sum_{1\leq |\alpha|\leq 2}\|\{\mathbf{I}-\mathbf{P}\}\partial^\alpha_{x}\delta f(\tau)\|_{\nu}^2
 + \|\delta(b-u)(\tau)\|_{\dot H^1\cap\dot H^2}^2\Big
 )\,{\rm d}\tau\nonumber\\
& \qquad \leq \delta \mathcal{X }(0)+C\varepsilon^2+C\big(\varepsilon_0^{\frac{1}{2}}+\varepsilon_1^{\frac{1}{2}}\big) \de\mathcal{X}(t),
\end{align}
where $C > 0 $ is a constant independent of $\varepsilon$ and time $t$.
\end{lem}

\begin{proof}
Applying the operator $\partial^\alpha_{x}$ with $1\leq|\alpha| \leq 2$ to \eqref{G5.2}$_1$--\eqref{G5.2}$_3$, respectively, and performing a basic $L^2$ energy estimate, we have
\begin{align}\label{G5.11}
&\frac{1}{2}  \Big(  \Big\|\frac{\sqrt{P^{\prime}(1+\varepsilon q^\varepsilon)}}{1+\varepsilon q^\varepsilon} \partial_{x}^\alpha \delta q(t)\Big\|_{L^2}^2+     \|\partial^\alpha_{x}\delta u(t)\|_{L^2}^2+\|\partial^\alpha_{x}\delta f(t)\|_{L_v^2(L^2)}^2   \Big)   \nonumber\\
& +\lambda_0\int_0^t \sum_{1\leq |\alpha|\leq 2}\|\{\mathbf{I}-\mathbf{P}\}\partial^\alpha_{x}\delta f(\tau)\|_{\nu}^2\,{\rm d}\tau + \int_0^t\|\partial_{x}^\alpha\delta(b-u)(\tau)\|_{L^2}^2\,{\rm d}\tau \nonumber\\
\leq\,&  \frac{1}{2}  \Big(  \Big\|\frac{\sqrt{P^{\prime}(1+\varepsilon q^\varepsilon)}}{1+\varepsilon q^\varepsilon} \partial_{x}^\alpha \delta q(0)\Big\|_{L^2}^2+     \|\partial^\alpha_{x}\delta u(0)\|_{L^2}^2+\|\partial^\alpha_{x}\delta f(0)\|_{L_v^2(L^2)}^2   \Big) \nonumber\\
& +\frac{1}{2}\int_0^t\!\!\int_{\mathbb R^3} \partial_{t} \Big( \frac{P^\prime(1+\varepsilon q^\varepsilon)}{(1+\varepsilon q^\varepsilon)^2}  \Big) |\partial^\alpha_{x}\delta q|^2 \,{\rm d}x{\rm d}\tau+\frac{1}{2}\int_0^t\!\!\int_{\mathbb R^3} \nabla_{x
}\cdot \Big( \frac{P^\prime(1+\varepsilon q^\varepsilon)}{(1+\varepsilon q^\varepsilon)^2} u^\varepsilon \Big) |\partial^\alpha_{x}\delta q|^2\,{\rm d}x{\rm d}\tau \nonumber\\
&-\int_0^t\!\! \int_{\mathbb R^3} \frac{P^\prime(1+\varepsilon q^\varepsilon)} {(1+\varepsilon q^\varepsilon)^2} \big([\partial^\alpha_{x}, u^\varepsilon\cdot\nabla_{x}]\delta q 
+[\partial^\alpha_{x
},  q^\varepsilon \nabla_{x
}\cdot ]\delta u  \big)\partial^\alpha_{x}\delta q \,{\rm d}x{\rm d}\tau \nonumber\\
&+\frac{1}{\varepsilon }\int_0^t\!\!\int_{\mathbb R^3} \nabla_{x}\bigg( \frac{P^\prime(1+\varepsilon q^\varepsilon)}{1+\varepsilon q^\varepsilon} \bigg)\partial^\alpha_{x} \delta q  \cdot \partial^\alpha_{x} \delta u  \,{\rm d}x{\rm d}\tau+\frac{1}{2}\int_0^t\!\!\int_{\mathbb R^3} |\partial^\alpha_{x} \delta u |^2{\rm div}_x  u^\varepsilon\,{\rm d}x{\rm d}\tau\nonumber\\
&-\int_0^t\!\!\int_{\mathbb R^3} \bigg(  [\partial^\alpha_{x},u^\varepsilon\cdot\nabla_{x}]\delta u+\frac{1}{\varepsilon} \Big[\partial^\alpha_{x}, \frac{P^\prime(1+\varepsilon q^\varepsilon)}{1+\varepsilon q^\varepsilon} \Big] \nabla_{x}\delta q         \bigg)\cdot \partial^\alpha_{x}\delta u  \,{\rm d}x{\rm d}\tau \nonumber\\
&+\int_0^t \!\!\int_{\mathbb R^3}\!\!\int_{\mathbb R^3}\partial^\alpha_{x}\Big( \frac{1}{2}u^\varepsilon\cdot v(\delta f)-u^\varepsilon\cdot\nabla_{v}\delta f  \Big)\partial^\alpha_{x}\delta f \,{\rm d}x{\rm d}v{\rm d}\tau-\int_0^t\!\!\int_{\mathbb R^3} \partial^\alpha_{x}\big(a(\delta u)\big)\cdot\partial^\alpha_{x}\delta u\,{\rm d}x{\rm d}\tau\nonumber\\
&+\int_0^t \!\!\int_{\mathbb R^3}\!\!\int_{\mathbb R^3}\partial^\alpha_{x}\Big( \frac{1}{2}(\delta u) \cdot v   f-\delta u\cdot\nabla_{v} f  \Big)\partial^\alpha_{x}\delta f \,{\rm d}x{\rm d}v{\rm d}\tau-\int_0^t\!\!\int_{\mathbb R^3 }\partial^\alpha_{x}\big((\delta a) u^\varepsilon\big)\cdot\partial^\alpha_{x}\delta u\,{\rm d}x{\rm d}\tau\nonumber\\
&+\int_0^t \!\!\int_{\mathbb R^3} \partial^\alpha_{x}(\delta q)\partial^\alpha_{x}\delta F_1\, {\rm d}x{\rm d}\tau+\int_0^t \!\!\int_{\mathbb R^3} \partial^\alpha_{x}(\delta u)\cdot \partial^\alpha_{x}(\delta F_2+\delta F_3)\, {\rm d}x{\rm d}\tau\nonumber\\
\equiv:\,&  \frac{1}{2}  \bigg(  \Big\|\frac{\sqrt{P^{\prime}(1+\varepsilon q^\varepsilon)}}{1+\varepsilon q^\varepsilon} \partial_{x}^\alpha \delta q(0)\Big\|_{L^2}^2+     \|\partial^\alpha_{x}\delta u(0)\|_{L^2}^2+\|\partial^\alpha_{x}\delta f(0)\|_{L_v^2(L^2)}^2   \bigg) +\sum_{i=1}^{12}\delta  J_{i}.
\end{align}

For the term $\delta J_1$, arguing analogously to \eqref{G3.11}, we have
\begin{align}\label{G5.12}
\delta J_1\lesssim&\,\varepsilon\int_0^t  \|q^\varepsilon\|_{H^3} \|\partial_t q^\varepsilon\|_{L^\infty} \|\partial^\alpha_{x}\delta q\|_{L^2}^2 \,{\rm d} \tau\nonumber\\
\lesssim&\,\varepsilon \|q^\varepsilon\|_{L_t^\infty(H^3)}\|\nabla_{x}\delta q\|_{L_t^2(H^1)}^2\nonumber\\
\lesssim&\, \varepsilon_0^\frac{1}{2}\delta\mathcal{X}(t).
\end{align}
For the  terms $\delta J_2,\dots,\delta J_6$, using
Lemmas \ref{L2.1}--\ref{L2.2}, H\"{o}lder's and Young's inequalities, we obtain 
\begin{align}\label{G5.13}
\sum_{i=2}^6 \delta J_i\lesssim &\,  \|\nabla_{x}(q^\varepsilon,u^\varepsilon )\|_{L_t^\infty(H^1)}^2\|\nabla_{x}\delta q\|_{L_t^2(L^2)}^2+\|(u^\varepsilon,q^\varepsilon)\|_{L_t^\infty(H^3)}\|\nabla_{x}(\delta u,\delta q)\|_{L_t^2(H^1)}^2 \nonumber\\
&+\frac{1}{\varepsilon}\| \varepsilon\nabla_{x} q^\varepsilon \|_{L_t^\infty(H^2)}\|\nabla_{x}(\delta u,\delta q)\|_{L_t^2(H^1)}^2\nonumber\\
\lesssim&\, \varepsilon_0^\frac{1}{2}\delta\mathcal{X}(t).
\end{align}
For the terms $\delta J_{7},\dots,\delta J_{10}$, applying the macro-micro decomposition \eqref{G2.1} and Lemma \ref{L2.1}, we deduce that
\begin{align}\label{G5.14}
\sum_{i=7}^{10}\delta J_{i}\lesssim&\, \|\nabla_{x
}u^\varepsilon\|_{L_t^2(H^2)}\|\delta f\|_{L_t^\infty(L_v^2(H^2))}\Big(  \|\nabla_{x}(\delta a,\delta b)\|_{L_t^2(H^1)}+ \sum_{1\leq |\alpha|\leq 2}\|\{\mathbf{I}-\mathbf{P}\}\partial^\alpha_{x}\delta f\|_{L_t^2(\nu)}          \Big)\nonumber\\
&+\|\nabla_{x
}\delta u \|_{L_t^2(H^1)}\|  f\|_{L_t^\infty(L_v^2(H^3))}\Big(  \|\nabla_{x}(\delta a,\delta b)\|_{L_t^2(H^1)}+ \sum_{1\leq |\alpha|\leq 2}\|\{\mathbf{I}-\mathbf{P}\}\partial^\alpha_{x}\delta f\|_{L_t^2(\nu)}          \Big) \nonumber\\
&+\|(a,u^\varepsilon)\|_{L_t^\infty(H^2)}\|\nabla_{x}(\delta u,\delta a)\|_{L_t^2(L^2)}^2\nonumber\\
\lesssim&\, \big(\varepsilon^\frac{1}{2}_0+\varepsilon_1^\frac{1}{2}\big) \delta\mathcal{X}(t).
\end{align}
Finally, for the remaining terms $\delta J_{11}$ and $\delta J_{12}$, by virtue of Lemma \ref{L2.1}, Proposition \ref{prop3} and Young's inequality, it holds
\begin{align}\label{G5.15}
\delta J_{11}+\delta J_{12}\lesssim &\, \varepsilon \|\nabla_{x}\delta q\|_{L_t^2(H^1)} \big( \|\nabla_{x}D_{t}\pi\|_{L_t^2(H^1)}+  \|\nabla_{x}\delta u\|_{L_t^2(H^1)}\|\nabla\pi\|_{L_t^\infty(H^2)}\big)+\|u\|_{L_t^\infty(H^3)}\|\nabla_{x}\delta u\|_{L_t^2(H^1)}^2\nonumber\\
&+\varepsilon \|q^\varepsilon\|_{L_t^\infty(H^3)}\big(\|b-u\|_{L_t^2(L^2)}+\|a\|_{L_t^\infty(L^3)}\|\nabla_{x} b\|_{L_t^2(L^2)}\big)\|\nabla_{x}\delta u\|_{L_t^2(H^1)}\nonumber\\
&+\varepsilon\|q^\varepsilon\|_{L_t^\infty(H^3)}\big(\|\delta(b-u)\|_{L_t^2(H^2)}\|\nabla_{x}\delta u\|_{L_t^2(H^1)}+\|\nabla_{x}(\delta u,\delta a)\|_{L_t^2(L^2)}^2\|(a,u^\varepsilon)\|_{L_t^\infty(H^3)} \big)\nonumber\\
&+\varepsilon\|\nabla_{x}\pi\|_{L_t^2(H^2)}\|\nabla_{x}q^\varepsilon\|_{L_t^2(H^1)}\|\delta u\|_{L_t^\infty(H^2)}\nonumber\\
\lesssim&\,\big( 1+\varepsilon_0^\frac{1}{2}+\varepsilon_1^\frac{1}{2}   \big)\delta\mathcal{X}(t)+\varepsilon^2+(\varepsilon_0+\varepsilon_1)\delta\mathcal{X}(t).
\end{align}

Substituting the estimates \eqref{G5.12}--\eqref{G5.15} into \eqref{G5.11} and then summing over $1\leq|\alpha|\leq 2$, we arrive at \eqref{G5.10}.
\end{proof}

Next, we show the $H^1$-norm estimate of $\nabla_{x}(\delta a,\delta b)$.
\begin{lem}\label{L5.3}
It holds that
\begin{align}\label{G5.16}
\int_0^t \|\nabla_{x}(\delta a,\delta b)(\tau)\|_{H^1}^2\,{\rm d}\tau\leq  \delta \mathcal{X }(0)+C\varepsilon^2+  C\delta\mathcal{X}(t),  
\end{align}
where $C > 0$ is a constant independent of $\varepsilon$ and time $t$.
\end{lem}

\begin{proof}
Similar to   \eqref{AB}, we   derive the equations of $\delta a$ and $\delta b$ as follows:
\begin{equation} \label{CD}
\left\{\begin{aligned}
&\partial_{t}\delta a +{\rm div}_x  \delta b=0,\\
&\partial_{t} \delta b _i+\partial_{i} \delta a +\sum_{j=1}^3\partial_j\Gamma_{ij}(\{\mathbf{I}-\mathbf{P}\}\delta f )= \delta u _i-\delta b _i +(\delta u_{i})a^\varepsilon+u_i(\delta a),  \\
&\partial_{i}\delta b _j+\partial_j \delta b_i- \big((\delta u _i)b^\varepsilon_j+u_i(\delta b_{j})+(\delta u _j)b^\varepsilon_i+u_{j}(\delta b_{i})\big)=-\partial_t \Gamma_{ij}(\{\mathbf{I}-\mathbf{P}\}\delta f)+\Gamma_{ij}(\delta \ell + \delta r  ),  
 \end{aligned}
 \right.
\end{equation}
for $1\leq i,j\leq 3$, where $\delta \ell$ and $\delta r$ are given by
\begin{align*}
\delta \ell :=&\, \mathcal{L}\{\mathbf{I}-\mathbf{P}\}\delta f -v\cdot\nabla_{x}\{\mathbf{I}-\mathbf{P}\}\delta f ,    \\
\delta r :=&\,  -\delta u\cdot\nabla_v\{\mathbf{I}-\mathbf{P}\}f^\varepsilon-u\cdot\nabla_{v}\{\mathbf{I}-\mathbf{P}\}\delta f+\frac{1}{2}\delta u \cdot v\{\mathbf{I}-\mathbf{P}\}f^\varepsilon+\frac{1}{2}u\cdot v\{\mathbf{I}-\mathbf{P}\}\delta f.
\end{align*}
We first compute the dissipation associated with  $\delta b$. Letting $|\alpha|\leq 1$, we have
\begin{align}\label{G5.18}
&2\|\nabla_{x}\partial^\alpha_{x}(\delta b)\|_{L^2}^2+2\|{\rm div}_x  \partial^\alpha_{x}(\delta  b)\|_{L^2}^2\nonumber\\
=&\,\sum_{i,j=1}^3 \|\partial^\alpha_{x}(\partial_{i}\delta b_{j}+\partial_{j}\delta b_{i})\|_{L^2}^2\nonumber\\
=&\,-\frac{{\rm d}}{{\rm d}t}\sum_{i,j=1}^3\int_{\mathbb R^3} \partial^\alpha_{x}(\partial_{i}\delta b_{j}+\partial_{j}\delta b_{i}) \partial^\alpha_{x}\Gamma_{ij}(\{\mathbf{I}-\mathbf{P}\}\delta f)\,{\rm d}x\nonumber\\
&\,+\sum_{i,j=1}^3\int_{\mathbb R^3} \partial^\alpha_{x}(\partial_{i}\partial_{t}\delta b_{j}+\partial_{j}\partial_{t}\delta b_{i}) \partial^\alpha_{x}\Gamma_{ij}(\{\mathbf{I}-\mathbf{P}\}\delta f)\,{\rm d}x\nonumber\\
&\,+ \sum_{i,j=1}^3\int_{\mathbb R^3} \partial^\alpha_{x}(\partial_{i}\delta b_{j}+\partial_{j}\delta b_{i}) \partial^\alpha_{x}\big((\delta u _i)b^\varepsilon_j+u_i(\delta b_{j})+(\delta u _j)b^\varepsilon_i+u_{j}(\delta b_{i}) +\Gamma_{ij}(\delta \ell + \delta r  )         \big) \,{\rm d}x \nonumber\\
\equiv:&\,-\frac{{\rm d}}{{\rm d}t}\sum_{i,j=1}^3\int_{\mathbb R^3} \partial^\alpha_{x}(\partial_{i}\delta b_{j}+\partial_{j}\delta b_{i}) \partial^\alpha_{x}\Gamma_{ij}(\{\mathbf{I}-\mathbf{P}\}\delta f)\,{\rm d}x+\delta K_1+\delta K_2.
\end{align}
Replacing the time derivative of $\delta b$ by \eqref{CD}$_2$, we obtain
\begin{align}\label{G5.19}
\delta K_1=\,&-2 \sum_{i,j=1}^3 \int_{\mathbb R^3}\partial^\alpha_{x} \partial_{t}(\delta b_{i})\partial^\alpha_{x}\partial_{j}\Gamma_{ij}(\{\mathbf{I}-\mathbf{P}\}\delta f)\,{\rm d}x \nonumber\\
=\,& 2\sum_{i,j=1}^3 \int_{\mathbb R^3}\partial_{x}^\alpha \Big(    \partial_{i} \delta a +\sum_{m=1}^3\partial_m\Gamma_{im}(\{\mathbf{I}-\mathbf{P}\}\delta f )\nonumber\\ 
& \quad  \qquad    -(\delta u-\delta b)_{i} -(\delta u_i)a^\varepsilon-u_{i}(\delta a)      \Big) \partial^\alpha_{x}\Gamma_{ij}(\{\mathbf{I}-\mathbf{P}\}\delta f)\,{\rm d}x\nonumber\\
\leq\,& \kappa\|\nabla_{x}\delta a\|_{H^1}^2+C \|\nabla_{x}\{\mathbf{I}-\mathbf{P}\}\delta f\|_{L_v^2(H^1)}^2\nonumber\\
& +C\big( \|\delta (b-u)\|_{H^1}^2+ \|(a^\varepsilon,u)\|_{H^1}^2\|\nabla_{x}(\delta a,\delta u)\|_{H^1}^2     \big)\nonumber\\
\leq\,& \kappa\|\nabla_{x}\delta a\|_{H^1}^2+C\big(   \|\nabla_{x}\{\mathbf{I}-\mathbf{P}\}\delta f\|_{L_v^2(H^1)}^2+\|\delta(b-u)\|_{H^1}^2+\|\nabla_{x}(\delta a,\delta u)\|_{H^1}^2  \big),
\end{align}
where we have utilized Young's inequality and Lemma \ref{L2.1}.
Here and below, $0<\kappa<1$ is a sufficiently small constant. 

Since the function appearing in $\Gamma_{ij}(\cdot)$ can absorb any velocity derivative and any velocity weight, we conclude that
\begin{align}\label{G5.20}
\delta K_2\leq&\, \frac{1}{2} \sum_{i,j=1}^3 \|\partial^\alpha_{x}(\partial_{i}\delta b_{j}+\partial_{j}\delta b_{i})\|_{L^2}^2+C\|\nabla_{x}(\delta u,\delta b)\|_{H^1}^2\|(u,b^\varepsilon)\|_{H^1}^2+C\|\nabla_{x}\delta u\|_{H^1}^2\|f^\varepsilon\|_{ L_v^2(H^2)}^2\nonumber\\
&+C\|u\|_{H^2}^2\|\nabla_{x}\{\mathbf{I}-\mathbf{P}\}\delta f\|_{L_v^2(H^1)}^2\nonumber\\
\leq&\, \frac{1}{2} \sum_{i,j=1}^3 \|\partial^\alpha_{x}(\partial_{i}\delta b_{j}+\partial_{j}\delta b_{i})\|_{L^2}^2+C\Big( \|\nabla_{x}(\delta u,\delta a,\delta b)\|_{H^1}^2+\sum_{1\leq|\alpha|\leq 2}\|\{\mathbf{I}-\mathbf{P}\}\partial^\alpha_{x}\delta f\|_{\nu}^2    \Big).
\end{align}
Putting the estimates \eqref{G5.19}--\eqref{G5.20} into \eqref{G5.18} gives rise to
\begin{align}\label{G5.21}
&\frac{{\rm d}}{{\rm d}t}\sum_{|\alpha|\leq 1}\sum_{i,j=1}^3\int_{\mathbb R^3} \partial^\alpha_{x}(\partial_{i}\delta b_{j}+\partial_{j}\delta b_{i}) \partial^\alpha_{x}\Gamma_{ij}(\{\mathbf{I}-\mathbf{P}\}\delta f)\,{\rm d}x+ (1+ \lambda_1^\prime)\|\nabla_{x} \delta b\|_{H^1}^2 \nonumber\\
&\quad \leq \kappa\|\nabla_{x}\delta a\|_{H^1}^2+C\Big( \|\nabla_{x}(\delta u,\delta a,\delta b)\|_{H^1}^2+\|\delta(b-u)\|_{H^1}^2+\sum_{ |\alpha|\leq 2}\|\{\mathbf{I}-\mathbf{P}\}\partial^\alpha_{x}\delta f\|_{\nu}^2    \Big),
\end{align}
for some constant $\lambda_{1}^\prime>0$.

Now, we focus on the dissipation of $\delta a$.
By direct calculations, we have
\begin{align}\label{G5.22}
\|\partial^\alpha_{x} \nabla_{x}\delta a\|_{L^2}^2=&\,\sum_{i=1}^3 \int_{\mathbb R^3} \partial^\alpha_{x} \partial_{i}(\delta a) \partial^\alpha_{x} \partial_{i}(\delta a)\, {\rm d}x    \nonumber\\
=&\, -\frac{{\rm d}}{{\rm d}t}\sum_{i=1}^3\int_{\mathbb R^3}\partial^\alpha_{x} \partial_{i}(\delta a)\partial^\alpha_{x}(\delta b_{i})\,{\rm d}x+\sum_{i=1}^3\int_{\mathbb R^3}\partial^\alpha_{x} \partial_{i}\partial_{t}(\delta a)\partial^\alpha_{x}(\delta b_{i})\,{\rm d}x\nonumber\\
&+ \sum_{i=1}^3 \int_{\mathbb R^3} \partial^\alpha_{x} \partial_{i}(\delta a)\partial^\alpha_{x}\Big((\delta u_{i}-\delta b_{i})-\sum_{j=1}^3\partial_{j}\Gamma_{ij}(\{\mathbf{I}-\mathbf{P}\}\delta f) +(\delta u_{i})a^\varepsilon+u_i(\delta a)\Big)\,{\rm d}x\nonumber\\
\equiv:&\,-\frac{{\rm d}}{{\rm d}t}\sum_{i=1}^3\int_{\mathbb R^3}\partial^\alpha_{x} \partial_{i}(\delta a)\partial^\alpha_{x}(\delta b_{i})\,{\rm d}x+\delta K_3+\delta K_4.
\end{align}
Using the equation \eqref{CD}$_1$ and integration by parts, we have
\begin{align}\label{G5.23}
\delta K_3=-\int_{\mathbb R^3} \partial^\alpha_{x} \partial_{t}(\delta a)\partial^\alpha_{x} {\rm div}_x  (\delta b)\,{\rm d}x=\|\partial^\alpha_{x} {\rm div}_x  (\delta b)\|_{L^2}^2.   
\end{align}
Thanks to Young’s inequality and Lemma \ref{L2.1}, we derive  that 
\begin{align}\label{G5.24}
\delta K_4\leq&\,  \frac{1}{2} \|\nabla_{x} \partial^\alpha_{x}(\delta a)\|_{L^2}^2+C\|\delta (b-u)\|_{H^1}^2+\|\nabla_{x}\{\mathbf{I}-\mathbf{P}\}\delta f\|_{L_v^2(H^1)}^2+C\|\nabla_{x}(\delta a,\delta u)\|_{H^1}^2\|(a^\varepsilon,u)\|_{H^1}^2\nonumber\\
\leq&\,  \frac{1}{2} \|\nabla_{x} \partial^\alpha_{x}(\delta a)\|_{L^2}^2+C\Big( \|\nabla_{x}(\delta u,\delta a,\delta b)\|_{H^1}^2+\|\delta(b-u)\|_{H^1}^2+\sum_{ |\alpha|\leq 2}\|\{\mathbf{I}-\mathbf{P}\}\partial^\alpha_{x}\delta f\|_{\nu}^2    \Big).
\end{align}
Inserting the estimates \eqref{G5.23}--\eqref{G5.24} into \eqref{G5.22} yields
\begin{align}\label{G5.25}
&-\frac{{\rm d}}{{\rm d}t} \sum_{|\alpha|\leq 1}\int_{\mathbb R^3} \partial^\alpha_{x} (\delta a)\partial^\alpha_{x}{\rm div}_x (\delta b)\,{\rm d}x +\frac{1}{2}\|\nabla_{x}(\delta a)\|_{H^1}^2 \nonumber\\ 
&\quad  \leq\|\nabla_{x}(\delta b)\|_{H^1}^2+C\Big( \|\nabla_{x}(\delta u,\delta a,\delta b)\|_{H^1}^2+\|\delta(b-u)\|_{H^1}^2+\sum_{ |\alpha|\leq 2}\|\{\mathbf{I}-\mathbf{P}\}\partial^\alpha_{x}\delta f\|_{\nu}^2    \Big).
\end{align}

Finally, we introduce the new temporal function $ \mathfrak{E}_0(t)$, given by
\begin{align*}
\mathfrak{E}_0(t):= \sum_{|\alpha|\leq 1}\sum_{i,j=1}^3\int_{\mathbb R^3} \partial^\alpha_{x}(\partial_{i}\delta b_{j}+\partial_{j}\delta b_{i}) \partial^\alpha_{x}\Gamma_{ij}(\{\mathbf{I}-\mathbf{P}\}\delta f)\,{\rm d}x-    \sum_{|\alpha|\leq 1}\int_{\mathbb R^3} \partial^\alpha_{x} (\delta a)\partial^\alpha_{x}{\rm div}_x (\delta b)\,{\rm d}x.
\end{align*}
It is easy to observe that
\begin{align}\label{G5.26}
|\mathfrak{E}_0(t)|\lesssim \|(\delta a,\delta b)\|_{H^2}^2 + \|\{\mathbf{I}-\mathbf{P}\}\delta f\|_{L_v^2(H^2)}^2\lesssim \|\delta f\|_{L_v^2(H^2)}^2 .  
\end{align}
Combining \eqref{G5.21} and  \eqref{G5.25}, we further obtain
\begin{align*} 
\frac{{\rm d}}{{\rm d}t}\mathfrak{E}_0(t)+\lambda_{2}^\prime \|\nabla_{x
}(\delta a,\delta b)\|_{H^1}^2\lesssim  C\Big( \|\nabla_{x}(\delta u,\delta a,\delta b)\|_{H^1}^2+\|\delta(b-u)\|_{H^1}^2+\sum_{ |\alpha|\leq 2}\|\{\mathbf{I}-\mathbf{P}\}\partial^\alpha_{x}\delta f\|_{\nu}^2    \Big), 
\end{align*}
for some constant $\lambda_{2}^\prime>0$,
which implies that 
\begin{align}\label{G5.27}
\mathfrak{E}_0(t) +\lambda_{2}^\prime\int_0^t\|\nabla_{x}(\delta a,\delta b)\|_{H^1}^2\,{\rm d}\tau\leq    \mathfrak{E}_0(0)+ C\delta\mathcal{X}(t).
\end{align}
Thus, based on the estimates \eqref{G5.3}, \eqref{G5.10}, \eqref{G5.26} and \eqref{G5.27}, we get \eqref{G5.16}. This completes the proof of Lemma \ref{L5.3}.
\end{proof}

Finally, we give the dissipation of $\delta q$.
\begin{lem}\label{L5.4}
It holds that
\begin{align}\label{G5.28}
\int_0^t \|\nabla_{x}(\delta q)(\tau)\|_{H^1}^2\,{\rm d}\tau\leq  \delta \mathcal{X }(0)+C\varepsilon^2+  C\delta\mathcal{X}(t),  
\end{align}
where $C > 0$ is a constant independent of $\varepsilon$ and time $t$.
\end{lem}
\begin{proof}
Let $|\alpha|\leq 1$. Applying the operator $\partial^\alpha_{x}$ to \eqref{G5.2}$_2$ and then performing a basic $L^2$ energy estimate, we have
\begin{align}\label{G5.29}
&{P^\prime(1)}\|\nabla_{x}\partial^\alpha_{x}(\delta q)\|_{L^2}^2\nonumber\\
=\, &-{\varepsilon}\int_{\mathbb{R}^3}\nabla_{x}\partial^\alpha_{x}(\delta q)\cdot\partial^\alpha_x\partial_t (\delta u) \,\mathrm{d}x+{\varepsilon}\int_{\mathbb{R}^3}\nabla_{x}\partial^\alpha_{x}(\delta q)\cdot\partial^\alpha_{x}\big(\delta(b-u)-a(\delta u)-(\delta a)u^\varepsilon\big)\,\mathrm{d}x   \nonumber\\
&-{\varepsilon}\int_{\mathbb{R}^3}\nabla_{x}\partial^\alpha_{x}(\delta q )\cdot\partial^\alpha_{x}\big(u^\varepsilon\cdot\nabla_{x} (\delta u) \big)\,\mathrm{d}x -\int_{\mathbb{R}^3}\nabla_{x}\partial^\alpha_{x
}(\delta q)\cdot\partial^\alpha_{x}\Big(\Big(\frac{P^{\prime}(1+\varepsilon q^\varepsilon)}{1+\varepsilon q^\varepsilon}-P^{\prime}(1)       \Big)\nabla_{x}(\delta q) \Big)\,\mathrm{d}x \nonumber\\ 
&+\varepsilon \int_{\mathbb R^3}\nabla_{x}\partial^\alpha_{x}(\delta q) \cdot \partial^\alpha_{x}(\delta F_2+\delta F_3)\,{\rm d}x\nonumber\\
\equiv:\,&\sum_{j=1}^5\delta L_{j}.
\end{align}
Applying \eqref{G5.2}$_1$, integration by parts and Young's inequality yields
\begin{align}\label{G5.30}
\delta L_1=& -\varepsilon\frac{{\rm d}}{{\rm d}t} \int_{\mathbb R^3}\nabla_{x}\partial^\alpha_{x}(\delta q)\cdot  \partial^\alpha_{x} (\delta u)\,{\rm d}x+\varepsilon\int_{\mathbb R^3}\nabla_{x}\partial^\alpha_{x}\partial_{t} (\delta q)\cdot\partial^\alpha_{x}(\delta u)\,{\rm d}x \nonumber\\
=&-\varepsilon\frac{{\rm d}}{{\rm d}t} \int_{\mathbb R^3}\nabla_{x}\partial^\alpha_{x}(\delta q)\cdot  \partial^\alpha_{x} (\delta u)\,{\rm d}x\nonumber\\
&+ \varepsilon\int_0^t\partial^\alpha_{x}\Big(u^\varepsilon\cdot\nabla_{x}(\delta q) +\frac{1+\varepsilon q^\varepsilon}{\varepsilon}\nabla_{x} \cdot(\delta u)-\delta F_1\Big){\rm div}_x \partial^\alpha_{x}(\delta u)\,{\rm d}x \nonumber\\
\leq\,& -\varepsilon\frac{{\rm d}}{{\rm d}t} \int_{\mathbb R^3}\nabla_{x}\partial^\alpha_{x}(\delta q)\cdot  \partial^\alpha_{x} (\delta u)\,{\rm d}x+ C(1+\|q^\varepsilon\|_{H^3})\|\nabla_{x}\delta u\|_{H^1}^2+ C\|u^\varepsilon\|_{H^3}^2\|\nabla_{x}(\delta q)\|_{H^1}^2\nonumber\\
&+ C\varepsilon^2\|\nabla_{x}D_{t}\pi\|_{H^1}^2+C\varepsilon^2 \|\nabla_{x
}\delta u\|_{H^1}^2 \|\nabla_{x}\pi\|_{H^2}^2 \nonumber\\
\leq& -\varepsilon\frac{{\rm d}}{{\rm d}t} \int_{\mathbb R^3}\nabla_{x}\partial^\alpha_{x}(\delta q)\cdot  \partial^\alpha_{x} (\delta u)\,{\rm d}x+C\varepsilon_0\|\nabla_{x}\delta q\|_{H^1}^2+C\|\nabla_{x}\delta u\|_{H^1}^2+C\varepsilon^2\|\nabla_{x}D_{t}\pi\|_{H^1}^2.
\end{align}
By means of H\"{o}lder’s and Young’s inequalities and Lemma \ref{L2.1}, we obtain the following estimate for the terms $\delta L_{2}, \dots, \delta L_5$:
\begin{align}\label{G5.31}
\delta L_2+\delta L_3+\delta L_4\leq&\, \frac{1}{3} P^\prime(1) \|\nabla_{x}\partial^\alpha_{x}(\delta q)\|_{L^2}^2+ C \big(  \|\delta (b-u)\|_{H^1}^2+\|\nabla_{x}(\delta a,\delta u)\|_{H^1}^2  \big) ,\\ \label{G5.32}
\delta L_{5}\leq&\, \frac{1}{3} P^\prime(1) \|\nabla_{x}\partial^\alpha_{x}(\delta q)\|_{L^2}^2+C\big(  \|\delta (b-u)\|_{H^1}^2+\|\nabla_{x}(\delta a,\delta b,\delta u)\|_{H^1}^2  \big) \nonumber\\
&+C\varepsilon^2 \big(\|\nabla_{x}\pi\|_{H^2}^2+\|b-u\|_{H^1}^2+\|\nabla_{x} u\|_{H^2}^2\big).
\end{align}
Putting the estimates \eqref{G5.30}--\eqref{G5.32} into \eqref{G5.29}, we end up with
\begin{align}\label{G5.33}
&\varepsilon\frac{{\rm d}}{{\rm d}t}\sum_{|\alpha|\leq 1} \int_{\mathbb R^3}\nabla_{x}\partial^\alpha_{x}(\delta q)\cdot  \partial^\alpha_{x} (\delta u)\,{\rm d}x  +\lambda^\prime
_{3}\|\nabla_{x}q\|_{H^1}^2\nonumber\\
&\quad \lesssim   \|\delta (b-u)\|_{H^1}^2+\|\nabla_{x}(\delta a,\delta b,\delta u)\|_{H^1}^2\nonumber\\
&\qquad  +\varepsilon^2  \big(\|\nabla_{x}D_{t}\pi\|_{H^1}^2+\|\nabla_{x}\pi\|_{H^2}^2+\|b-u\|_{H^1}^2+\|\nabla_{x} u\|_{H^2}^2\big),
\end{align}
for some constant $\lambda_{3}^\prime>0$.
Since
\begin{align*}
\varepsilon\Big|\sum_{|\alpha|\leq 1} \int_{\mathbb R^3}\nabla_{x}\partial^\alpha_{x}(\delta q)\cdot  \partial^\alpha_{x} (\delta u)\,{\rm d}x \Big|\lesssim \varepsilon\|\delta q\|_{H^2}^2+\varepsilon\|\delta u\|_{H^2}^2\lesssim \varepsilon\delta \mathcal{X}(t),    
\end{align*}
the time integration of \eqref{G5.33} combined with Proposition \ref{prop3} yields \eqref{G5.28}, and thus we complete the proof of Lemma \ref{L5.4}.
\end{proof}

With Lemmas \ref{L5.1}--\ref{L5.4} in hand, we proceed to prove Theorem \ref{Th4}.
\begin{proof}[Proof of Theorem \ref{Th4}]
It follows from \eqref{G5.3}, \eqref{G5.10}, \eqref{G5.16}, \eqref{G5.28} and the assumption \eqref{TD1} that
\begin{align*}
\delta\mathcal{X}(t)\lesssim  \big( \varepsilon_0^\frac{1}{2}+\varepsilon_1^\frac{1}{2}\big) \delta\mathcal{X}(t)+\varepsilon^2,   
\end{align*}
which, together with the smallness of $\varepsilon_0$ in \eqref{TG1} and $\varepsilon_1$ in \eqref{TGG1}, gives
\begin{align*}
  \delta\mathcal{X}(t)\lesssim \varepsilon^2.  
\end{align*}
Hence, the global estimate \eqref{TD2} is rigorously verified, thereby completing the proof of Theorem \ref{Th4}.
\end{proof}

\subsection{Proof of Theorem \ref{Th5}}
In this subsection, we prove the low Mach number limit of the compressible Euler--VFP system \eqref{A1}. The key point is that the error estimate \eqref{TD2} is directly formulated in a time-continuous topology. Consequently, once \eqref{TD2} is derived, the convergence occurs at an explicit rate that is uniform in \(C(\mathbb{R}^+;H^2)\). Specifically, it holds in \(C([0,T];H^2)\) for any \(T > 0\) with a rate independent of \(T\) and without any exponential growth factor \(e^{CT}\), and the limiting equations can be determined by taking the limit term by term. Notably, in contrast to the proof in \cite[Theorem 1.2]{LNW-2026}, no compactness argument, such as the Aubin-Lions's lemma, is required here.

\begin{proof}[Proof of Theorem \ref{Th5}]
First, we prove the strong convergence in \eqref{convergence} in terms of the error variables.
From Remark \ref{Rem1.2} and the estimate \eqref{TD2}, it follows that  
\begin{align*}
\sup_{t\geq 0}\|q^\varepsilon(t)\|_{H^2}\leq  C \sup_{t
\geq 0}\big\|\big(q^\varepsilon- \varepsilon {P^{\prime}(1)}^{-1} \pi\big)(t)\big\|_{H^2} + C\varepsilon\sup_{t\geq 0}\|\pi(t)\|_{H^2}\leq C\varepsilon,
\end{align*}
for some constant $C>0$ independent of time, which results in
\begin{align*}
 q^\varepsilon\rightarrow  0\quad \text{strongly~~~in}\quad C(\mathbb R^+;H^2)\quad \text{as}\quad \varepsilon\rightarrow0.
\end{align*}
Thus, we have
\begin{align*}
 \rho^\varepsilon-1\rightarrow 0\quad \text{strongly~~~in}\quad C(\mathbb R^+;H^2)\quad \text{as}\quad \varepsilon\rightarrow0.
\end{align*}
Due to the embedding $H^2(\mathbb{R}^3) \hookrightarrow L^\infty(\mathbb{R}^3)$, we further obtain 
\begin{align*}
 \rho^\varepsilon\rightarrow 1\quad \text{strongly~~~in}\quad L^\infty(\mathbb R^+;L^\infty)\quad \text{as}\quad \varepsilon\rightarrow0.
\end{align*}

For the convergence of $u^\varepsilon$ and $f^\varepsilon$, from \eqref{TD2}, we directly get  
\begin{align*}
  u^\varepsilon&\,\rightarrow u\quad \text{strongly~~~in}\quad C(\mathbb R^+;H^2)\quad\quad\,\,\,\,\, \text{as}\quad \varepsilon\rightarrow0,\\
    f^\varepsilon&\,\rightarrow f\quad \text{strongly~~~in}\quad C(\mathbb R^+;L_v^2(H^2))\quad \text{as}\quad \varepsilon\rightarrow0.
\end{align*}

For the convergence of $a^\varepsilon$ and $b^\varepsilon$, with the help of Cauchy--Schwarz's inequality, we have
\begin{align*}
\|(a-a^\varepsilon)(t)\|_{H^2}\lesssim &\, \Big(\int_{\mathbb R^3} M(v)\,{\rm d}v\Big)^\frac{1}{2}    \|(f^\varepsilon-f)(t)\|_{L_v^2(H^2)}\lesssim \|(f^\varepsilon-f)(t)\|_{L_v^2(H^2)}
\end{align*}
and 
\begin{align*}
\|(b-b^\varepsilon)(t)\|_{H^2}\lesssim &\, \Big(\int_{\mathbb R^3} |v|^2M(v)\,{\rm d}v\Big)^\frac{1}{2}    \|(f^\varepsilon-f)(t)\|_{L_v^2(H^2)}\lesssim \|(f^\varepsilon-f)(t)\|_{L_v^2(H^2)},
\end{align*}
which implies that
\begin{align*}
 (a^\varepsilon,b^\varepsilon)\rightarrow (a,b)\quad \text{strongly~~~in}\quad C(\mathbb R^+;H^2)\quad \text{as}\quad \varepsilon\rightarrow0.
\end{align*}

To prove the convergence of $\mathbb Qu^\varepsilon$ in \eqref{convergence}, we require the incompressibility condition ${\nabla_{x}}\cdot u = 0$. Its validity will be shown below. 
To avoid singularities, we multiply   \eqref{A1}$_1$ by $\varepsilon$, and then get
\begin{align}\label{G5.34}
\rho^\varepsilon {\rm div}_x  u^\varepsilon=-\varepsilon \partial_{t}q^\varepsilon-\varepsilon u^\varepsilon\cdot \nabla_{x}q^\varepsilon.    
\end{align}
For any $T > 0$ and   test function $\phi\in C_{c}^\infty((0,T)\times \mathbb R^3)$, using integration by parts and Cauchy--Schwarz's inequality, we have
\begin{align*}
\Big |\int_0^T \!\!\int_{\mathbb R^3}\varepsilon \partial_{t} q^\varepsilon \phi\,{\rm d}x{\rm d}t\Big |=&\,\Big|-\varepsilon\int_0^T\!\!\int_{\mathbb R^3} q^\varepsilon \partial_{t}\phi \,{\rm d}x{\rm d}t\Big |\lesssim \varepsilon\|q^\varepsilon\|_{L_t^\infty(L^2)} \|\partial_{t}\phi\|_{L_t^1(L^2)},
\end{align*}
which yields
\begin{align}\label{G5.35}
\int_0^T \!\!\int_{\mathbb R^3}\varepsilon \partial_{t} q^\varepsilon \phi\,{\rm d}x{\rm d}t\rightarrow 0 \quad \text{as}\quad   \varepsilon\rightarrow 0.  
\end{align}
For the second term on the right-hand side of \eqref{G5.34}, we have
\begin{align*}
\Big|\int_0^T \!\! \int_{\mathbb R^3}\varepsilon u^\varepsilon\cdot \nabla_{x}q^\varepsilon \phi\,{\rm d}x{\rm d}\tau\Big|\lesssim \varepsilon \|u^\varepsilon\|_{L_t^\infty(H^2)}\|q^\varepsilon\|_{L_t^\infty(H^2)}\|\phi\|_{L_t^1(L^2)},
\end{align*}
which leads to
\begin{align}\label{G5.36}
\int_0^T \!\! \int_{\mathbb R^3}\varepsilon u^\varepsilon\cdot \nabla_{x}q^\varepsilon \phi\,{\rm d}x{\rm d}\tau\rightarrow 0 \quad \text{as}\quad   \varepsilon\rightarrow 0.  
\end{align}
Combining \eqref{G5.35} and \eqref{G5.36}, we arrive at
\begin{align}\label{G5.37}
\rho^\varepsilon {\rm div}_x  u^\varepsilon\rightarrow 0\quad \text{in} \quad \mathcal{D}^\prime ((0,T)\times\mathbb R^3) \quad \text{as}\quad   \varepsilon\rightarrow 0.  
\end{align}

Let us return to the left-hand side of \eqref{G5.34} and decompose $\rho^\varepsilon{\rm div}_x  u^\varepsilon - {\rm div}_x  u$ as  
\begin{align*}
 \rho^\varepsilon{\rm div}_x  u^\varepsilon - {\rm div}_x  u=  (\rho^\varepsilon-1){\rm div}_x  u^\varepsilon+ {\rm div}_x (u^\varepsilon-u).
\end{align*}
Based on the fact that
\begin{align*}
&\Big|\int_0^T \!\! \int_{\mathbb R^3} \big( (\rho^\varepsilon-1){\rm div}_x  u^\varepsilon+ {\rm div}_x (u^\varepsilon-u)\big) \phi\,{\rm d}x{\rm d}\tau\Big|\nonumber\\
&\quad \lesssim \big(\|\rho^\varepsilon-1\|_{L_t^\infty(H^2)} \|u^\varepsilon\|_{L_t^\infty(H^2)}+\|u^\varepsilon-u\|_{L_t^\infty(H^2)}\big) \|\phi\|_{L_t^1(L^2)}  \rightarrow0 \quad \text{as}\quad \varepsilon\rightarrow0,
\end{align*}
we have
\begin{align*}
 \rho^\varepsilon {\rm div}_x  u^\varepsilon\rightarrow {\rm div}_x  u\quad \text{in} \quad \mathcal{D}^\prime ((0,T)\times\mathbb R^3) \quad \text{as}\quad   \varepsilon\rightarrow 0,
\end{align*}
which, together with \eqref{G5.37} and the arbitrariness of the test function $\phi$, yields
\begin{align}\label{G5.38}
 {\rm div}_x  u=0\quad \text{in} \quad \mathcal{D}^\prime ((0,T)\times\mathbb R^3),   
\end{align}
for any $T>0$. 
With the above preparations, it is easy to prove the convergence of $\mathbb Qu^\varepsilon$ in \eqref{convergence}.
In fact, since $\mathbb Q$ is a bounded zero-order Fourier multiplier on $\mathbb R^3$, making use of \eqref{G5.38}, we obtain
\begin{align*}
 \mathbb Q u^\varepsilon\rightarrow 0  \quad \text{strongly~~~in}\quad C(\mathbb R^+;H^2) \quad \text{as}\quad \varepsilon\rightarrow0.  
\end{align*}
Consequently, we complete the proof of \eqref{convergence}. 

Next, we show the convergence of  \eqref{A1}$_2$. We rewrite \eqref{A1}$_2$ as
\begin{align}\label{G5.39}
\nabla_{x}\Pi^\varepsilon:=\frac{1}{\varepsilon^2} \frac{P^\prime(\rho^\varepsilon)}{\rho^\varepsilon}\nabla_{x}\rho^\varepsilon=  -\partial_{t }u^\varepsilon-u^\varepsilon\cdot\nabla_{x}u^\varepsilon+\frac{b^\varepsilon-u^\varepsilon-a^\varepsilon u^\varepsilon}{\rho^\varepsilon}.
\end{align}
The definition of $\nabla_{x}\Pi^\varepsilon$ is valid. In fact, denoting $h^\prime(r)=\frac{P^\prime(r)}{r}$, we have
\begin{align*}
 \frac{1}{\varepsilon^2} \frac{P^\prime(\rho^\varepsilon)}{\rho^\varepsilon}\nabla_{x}\rho^\varepsilon=\nabla_{x}\Big(   \frac{h(\rho^\varepsilon)-h(1)}{\varepsilon^2}  \Big) .
\end{align*}
Then, to handle the pressure term, we apply the projection operator $\mathbb{P}$ to \eqref{G5.39} and obtain
\begin{align}\label{G5.40}
\partial_{t}\mathbb Pu^\varepsilon+\mathbb P(u^\varepsilon\cdot\nabla_{x} u^\varepsilon) =  \mathbb P \Big(  \frac{b^\varepsilon-u^\varepsilon-a^\varepsilon u^\varepsilon}{\rho^\varepsilon}  \Big).
\end{align}
For any test function $\Phi\in C_{c}^\infty((0,T)\times \mathbb R^3)$, it is easy to check that
\begin{align*}
\Big |\int_0^T \!\!\int_{\mathbb R^3} \partial_{t} \mathbb{P} (u^\varepsilon-u) \cdot \Phi\,{\rm d}x{\rm d}t\Big |=&\,\Big|-\int_0^T\!\!\int_{\mathbb R^3} \mathbb P (u^\varepsilon-u)\cdot \partial_{t}\Phi \,{\rm d}x{\rm d}t\Big |\nonumber\\
\lesssim&\,  \| u^\varepsilon-u\|_{L_t^\infty(L^2)} \|\partial_{t}\Phi\|_{L_t^1(L^2)}\rightarrow0 \quad \text{as} \quad \varepsilon\rightarrow0.
\end{align*}
Therefore, using \eqref{G5.38},  we derive that
\begin{align}\label{G5.41}
\partial_{t}\mathbb Pu^\varepsilon\rightarrow \partial_{t}u    \quad \text{in} \quad \mathcal{D}^\prime ((0,T)\times\mathbb R^3) \quad \text{as}\quad   \varepsilon\rightarrow 0. 
\end{align}
Similarly, we   easily deduce that
\begin{align}\label{G5.42}
\mathbb P(u^\varepsilon\cdot\nabla_{x} u^\varepsilon)\rightarrow\mathbb P(u\cdot\nabla_{x}u)    \quad \text{in} \quad \mathcal{D}^\prime ((0,T)\times\mathbb R^3) \quad \text{as}\quad   \varepsilon\rightarrow 0,
\end{align}
and
\begin{align}\label{G5.43}
\mathbb P \Big(  \frac{b^\varepsilon-u^\varepsilon-a^\varepsilon u^\varepsilon}{\rho^\varepsilon}  \Big)\rightarrow \mathbb P(b-u-au)   \quad \text{in} \quad \mathcal{D}^\prime ((0,T)\times\mathbb R^3) \quad \text{as}\quad   \varepsilon\rightarrow 0.
\end{align}
It follows from \eqref{G5.41}--\eqref{G5.43} that
\begin{align*}
\partial_{t} u+\mathbb{P}   (u\cdot\nabla_{x}u)= \mathbb P(b-u-au)   \quad \text{in} \quad \mathcal{D}^\prime ((0,T)\times\mathbb R^3),
\end{align*}
which gives rise to
\begin{align}\label{G5.44}
\partial_{t} u+    u\cdot\nabla_{x}u +\nabla_{x}\pi=   b-u-au    \quad \text{in} \quad \mathcal{D}^\prime ((0,T)\times\mathbb R^3).
\end{align}
for some function $\pi$.

Finally, we consider the kinetic equation \eqref{A1}$_3$.
For any test function $\psi\in C_{c}((0,T)\times\mathbb R^3\times\mathbb R^3)$, we   deduce that as $\varepsilon\rightarrow0$,
\begin{align*}
&\int_0^T \!\!\int_{\mathbb R^3}   \!\!\int_{\mathbb R^3}   \partial_{t}f^\varepsilon\psi\, {\rm d}x{\rm d}v{\rm d}t
=-\int_0^T \!\!\int_{\mathbb R^3}   \!\!\int_{\mathbb R^3}   f^\varepsilon\partial_{t}\psi\, {\rm d}x{\rm d}v{\rm d}t\nonumber\\
&\quad \rightarrow -\int_0^T \!\!\int_{\mathbb R^3}   \!\!\int_{\mathbb R^3}   f \partial_{t}\psi\, {\rm d}x{\rm d}v{\rm d}t=\int_0^T \!\!\int_{\mathbb R^3}   \!\!\int_{\mathbb R^3}   \partial_{t}f \psi\, {\rm d}x{\rm d}v{\rm d}t ,
\end{align*}
and
\begin{align*}
&\int_0^T \!\!\int_{\mathbb R^3}   \!\!\int_{\mathbb R^3}    v\cdot\nabla_{x}f^\varepsilon\psi\, {\rm d}x{\rm d}v{\rm d}t
=-\int_0^T \!\!\int_{\mathbb R^3}   \!\!\int_{\mathbb R^3}  f^\varepsilon v\cdot \nabla_{x}\psi\, {\rm d}x{\rm d}v{\rm d}t\nonumber\\
&\quad \rightarrow -\int_0^T \!\!\int_{\mathbb R^3}   \!\!\int_{\mathbb R^3}   f v\cdot \nabla_{x} \psi\, {\rm d}x{\rm d}v{\rm d}t=\int_0^T \!\!\int_{\mathbb R^3}   \!\!\int_{\mathbb R^3}  v\cdot\nabla_{x}f \psi\, {\rm d}x{\rm d}v{\rm d}t.
\end{align*}
Thus, we have
\begin{align}\label{G5.45}
\partial_{t}f^\varepsilon+v\cdot\nabla_{x}f^\varepsilon\rightarrow    \partial_{t}f +v\cdot\nabla_{x}f \quad \text{in} \quad \mathcal{D}^\prime ((0,T)\times\mathbb R^3\times\mathbb R^3) \quad \text{as}\quad   \varepsilon\rightarrow 0.
\end{align}
In addition, observe that
\begin{align*}
  &\Big|\int_0^T \!\!\int_{\mathbb R^3}   \!\!\int_{\mathbb R^3}  (f^\varepsilon u^\varepsilon-fu)\cdot \nabla_{v}\psi\, {\rm d}x{\rm d}v{\rm d}t \Big|\lesssim  \big(\|f^\varepsilon-f\|_{L_t^\infty(H^2)}\|u\|_{L_t^\infty(H^2)}
 \nonumber\\
  &\qquad\qquad\qquad +\|f^\varepsilon\|_{L_t^\infty(H^2)}\|u^\varepsilon-u\|_{L_t^\infty(H^2)}\big) \|\nabla_{v}\psi\|_{L_t^1(L_{v}^2(L^2))} \rightarrow 0 \quad \text{as}\quad \varepsilon\rightarrow0.
\end{align*}
Then, we arrive at
\begin{align}\label{G5.46}
u\cdot\nabla_{v}f^\varepsilon\rightarrow     u\cdot\nabla_{v}f \quad \text{in} \quad \mathcal{D}^\prime ((0,T)\times\mathbb R^3\times\mathbb R^3) \quad \text{as}\quad   \varepsilon\rightarrow 0.
\end{align}
Similarly, we get
\begin{align}\label{G5.47}
-\frac{1}{2}u^\varepsilon\cdot vf^\varepsilon-u^\varepsilon\cdot v\sqrt{M}\rightarrow  -\frac{1}{2}   u\cdot {v}f-u\cdot v\sqrt{M} \quad \text{in} \quad \mathcal{D}^\prime ((0,T)\times\mathbb R^3\times\mathbb R^3) \quad \text{as}\quad   \varepsilon\rightarrow 0.
\end{align}
For the last term containing the operator $\mathcal{L}$, we use its self-adjointness property to obtain
\begin{align*}
& \int_0^T \!\!\int_{\mathbb R^3}    \langle\mathcal{L} f^\varepsilon,\psi\rangle\,{\rm d}x {\rm d}t
= \int_0^T \!\!\int_{\mathbb R^3}    \langle f^\varepsilon,\mathcal{L}\psi\rangle\,{\rm d}x {\rm d}t 
\nonumber\\ 
&\quad \rightarrow  \int_0^T \!\!\int_{\mathbb R^3}    \langle f ,\mathcal{L}\psi\rangle\,{\rm d}x {\rm d}t= \int_0^T \!\!\int_{\mathbb R^3}    \langle\mathcal{L} f ,\psi\rangle\,{\rm d}x {\rm d}t\quad\text{as}\quad \varepsilon\rightarrow0,
\end{align*}
which yields
\begin{align}\label{G5.48}
 \mathcal{L}f^\varepsilon\rightarrow   \mathcal{L}f\quad \text{in} \quad \mathcal{D}^\prime ((0,T)\times\mathbb R^3\times\mathbb R^3) \quad \text{as}\quad   \varepsilon\rightarrow 0.    
\end{align}
Consequently, it follows from \eqref{G5.45}, \eqref{G5.46}, \eqref{G5.47} and \eqref{G5.48} that
\begin{align}\label{G5.49}
 \partial_t f +v\cdot\nabla_{x} f+ u \cdot\nabla_{v}f -\frac{1}{2}u \cdot v f -u \cdot v\sqrt{M}=\mathcal{L} f  \quad \text{in} \quad \mathcal{D}^\prime ((0,T)\times\mathbb R^3\times\mathbb R^3).    
\end{align}
Collecting \eqref{G5.38}, \eqref{G5.44} and \eqref{G5.49} together, we deduce that \eqref{A2} holds in the sense of distributions, and hence complete the proof of Theorem \ref{Th5}.
\end{proof}

\bigskip 
\noindent{\bf Acknowledgements:} 
Li and   Ni were supported by NSFC (Grant No. 12331007).
And Li was also supported by the ``333 Project" of Jiangsu Province. Zhang  was supported by NSFC (Grant No. 12471215) and Taishan Scholars Program (tsqn202507101).

\vspace{2mm}

\noindent\textbf{Conflict of interest.} The authors do not have any possible conflicts of interest.

\vspace{2mm}

\noindent\textbf{Data availability statement.}
 Data sharing is not applicable to this article as no data sets were generated or analyzed during the current study.

\bibliographystyle{plain}

\end{document}